\documentclass[12pt]{article}
\usepackage[english]{babel}
\usepackage{amsthm,amsfonts, amsbsy, amssymb,amsmath,graphicx}
\usepackage{graphics}

\newtheorem{theorem}{Theorem}
\newtheorem*{theorem*}{Theorem}

\newtheorem{corollary}{Corollary}

\newtheorem{proposition}{Proposition}
\newtheorem{propositionap}{Proposition}[section]
\theoremstyle{definition}
\newtheorem{definition}{Definition}
\newtheorem{definitionap}{Definition}[section]
\theoremstyle{remark}
\newtheorem{remark}{Remark}
\newtheorem{example}{Example}

\newtheorem{examplap}{Example}[section]
\newtheorem{remarkap}{Remark}[section]

\def\Z{{\mathbb Z}}
\def\N{{\mathbb N}}

\def\R{{\mathbb R}}

\newcommand{\V}{{\mathcal V}}

\date{}

\title{Crossing indices, traits and the principle of indistinguishability}
\author{Igor Nikonov\footnote{The author was supported by the Russian Foundation for Basic Research (grant No. 19-51-51004-NIF-a).}}

\begin{document}

\maketitle

\begin{abstract}
A (weak chord) index is a function on the crossings of knot diagrams such that: 1) the index of a crossing does not change under Reidemeister moves; 2) crossings which can be paired by a second Reidemeister move have the same index. We show that one can omit the second condition in the case of the universal index. As a consequence, we get the following principle of indistinguishability for classical knots: crossings of the same sign in a classical knot diagram can not be distinguished by any inherent property.
\end{abstract}

\section*{Introduction}

The conventional combinatorial approach to knot theory looks at knots as diagrams considered up to special transformations --- Reidemeister moves. Many constructions of knot invariants rely on combinatorial elements of diagrams --- arcs and crossings. But what if we could distinguish crossings of knot diagrams and divide them into classes and tribes? For example, discern a loop crossing which appears in a first Reidemeister move and a crossing which survived the reduction to the minimal diagram (Fig.~\ref{pic:crossing_hypotetic_types}).

\begin{figure}[h]
\centering\includegraphics[width=0.15\textwidth]{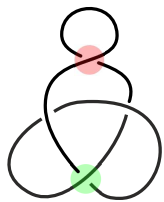}
\caption{Crossings in a knot diagram}\label{pic:crossing_hypotetic_types}
\end{figure}

For virtual knots, this idea was excellently realized by V.O. Manturov who created the parity theory~\cite{M3} and then applied it to strengthen knot invariants, to prove minimality
theorem and to construct (counter)examples~\cite{M4,M5,M6,M7}, see also~\cite{IM,IMN,IMN1,Nwp,Nif}.

Despite all achievements of the parity theory in the theory of virtual knots, it has one flow: the parity theory gives nothing to classical knots~\cite[Corollary 4.3]{IMN}. The reason is any parity on a knot is derived from the homology of the space the knot lies in, and the plane where the classical knots live, has trivial homology.

After this failure with parity we can try to weaken the parity axioms and consider another type of labelling of crossings --- index. The first appearance of an index can be found in V. Turaev's work~\cite{T} on strings (flat knots). After the paper of A. Henrich~\cite{Henrich}, the index appeared in works of several authors~\cite{Cheng,ILL,K2} and others. Z. Cheng~\cite{Cheng2} gave an axiomatic description of the index. M. Xu~\cite{Xu} found the most general formulation of index which he called \emph{weak chord index}.

The index polynomial has become a very efficient tool in the virtual knot theory. May be, it will be useful for study of classical knots. But, alas! There is no nontrivial indices on classical knots~\cite[Corollary 5]{Nct}.

This paper is a sequel of~\cite{Nct}. Here we make the final step and show that there can not be nontrivial invariant labels (except the sign, surely) on crossings of classical knot diagrams. This fact can be formulated as the following \emph{principle of indistinguishability}:
\begin{quote}
\em Crossings of the same sign in a classical knot diagram can not be distinguished by any inherent property.
\end{quote}

The mechanism behind the indistinguishability principle can be expressed as the \emph{substitution principle}:
\begin{quote}
\em Given two crossing of the same sign in a classical knot diagram, one of them can substitute for the other in the diagram.
\end{quote}
The exact formulation of the principles are given in Theorems~\ref{thm:indistinguishability_principle} and~\ref{thm:substitution_principle}.

A demonstration how the substitution principle works is shown in Fig.~\ref{pic:crossing_substitute}.

\begin{figure}[h]
\centering\includegraphics[width=0.9\textwidth]{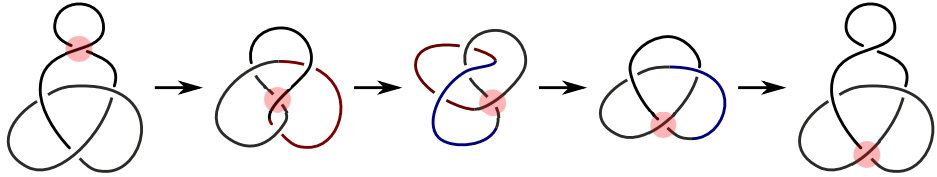}
\caption{A crossing substitutes for another one}\label{pic:crossing_substitute}
\end{figure}

The paper is organized as follows. In Section~\ref{sect:knot_theories} we define categories of tangle, link and knot diagrams. In Section~\ref{sect:indices} we introduce the notions of (signed) index, index polynomial and the universal index. Section~\ref{sect:traits} contains the central result of the paper: we show that the universal trait (the invariant labelling of crossing which contains the most information on them) is the universal signed index (Theorem~\ref{thm:main_theorem}). In Section~\ref{sect:principles} we combine the central result with the description of the universal index from~\cite{Nct} and formulate the principles of indistinguishability and substitution. We conclude the paper with the appendix where we list some known examples of index.

\section{Knot theories}\label{sect:knot_theories}

We start with definitions of virtual knots and related objects.

\begin{definition}\label{def:tangle}
Let $F$ be an oriented compact connected surface with the boundary $\partial F$ (perhaps empty).
An \emph{oriented tangle} in the thickened surface $F\times (0,1)$ is an embedding $T\colon M\to F\times (0,1)$ of an oriented compact $1$-dimensional manifold $M$ such that $T(\partial M)\subset \partial F\times (0,1)$ and $T$ is transversal to $\partial F\times(0,1)$. The image of the map $T$ will be also called a tangle.

A  connected component of a tangle whose boundary is empty is called a \emph{closed component}, otherwise, the components is called \emph{long}.

If a tangle consists of one closed component then it is a \emph{knot}; a tangle consisting of one long component is a \emph{long knot}; a tangle with several component which are all closed, is a \emph{link}.

We consider tangles up to isotopies of $F\times (0,1)$ with the boundary $\partial F\times (0,1)$ fixed.

A \emph{tangle diagram} $D$ is the image of a tangle $T$ in general position under the natural projection $F\times (0,1)\to F$. Combinatorially, a tangle diagram $D$ is an embedded graph in $F$ with vertices of valency $4$ (called \emph{crossings}) and vertices of valency $1$ (forming the \emph{boundary} $\partial D$ of the diagram $D$) such that $D\cap\partial F=\partial D$ and each vertex of valency $4$ carries the structure of under-overcrossing (Fig.~\ref{pic:tangle}).

\begin{figure}[h]
\centering
  \includegraphics[width=0.4\textwidth]{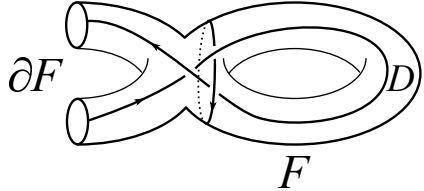}
  \caption{A tangle diagram}\label{pic:tangle}
\end{figure}
\end{definition}

Any tangle diagrams which correspond to isotopical tangles can be connected by a sequence of isotopies of the surface $F$ identical on $\partial F$, and \emph{Reidemeister moves} (Fig.~\ref{pic:reidmove}).

\begin{figure}[h]
\centering
  \includegraphics[width=0.8\textwidth]{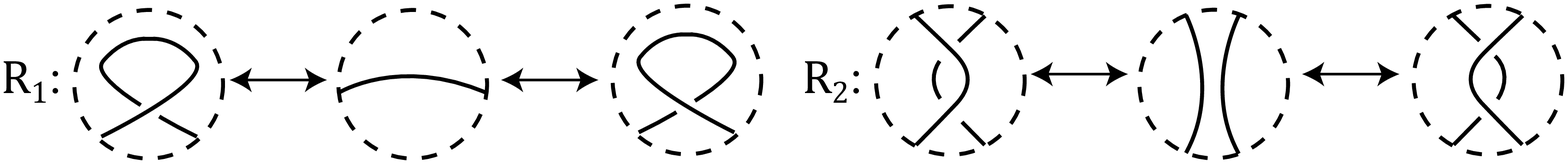}\\ \vbox{\phantom{1em}}
  \includegraphics[width=0.3\textwidth]{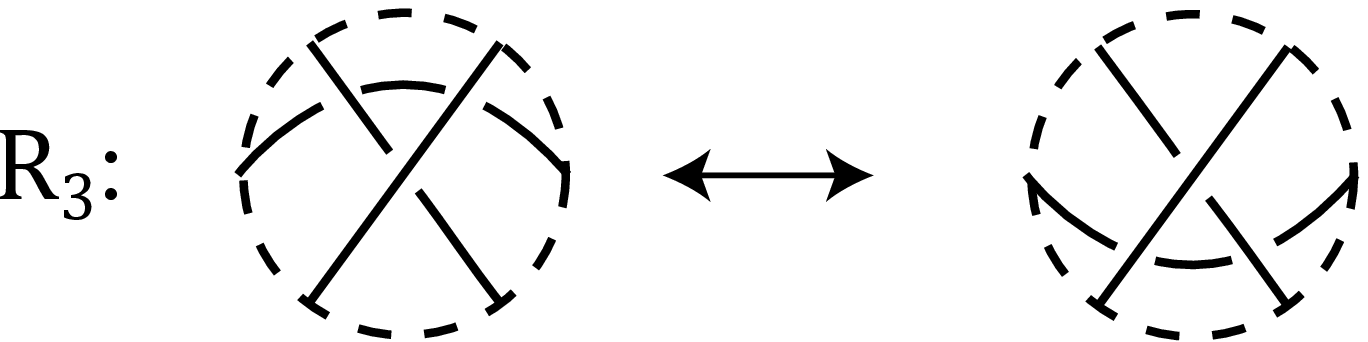}
  \caption{Reidemeister moves}\label{pic:reidmove}
\end{figure}

Let $D$ be a tangle diagram. Denote the set of its crossings by $\V(D)$. We assume that all tangle diagrams are oriented, hence, any crossing $v\in\V(D)$ has a \emph{sign} (Fig.~\ref{pic:crossing_sign}).

\begin{figure}[h]
\centering\includegraphics[width=0.25\textwidth]{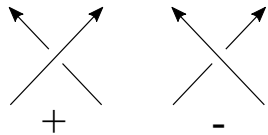}
\caption{The sign of a crossing}\label{pic:crossing_sign}
\end{figure}

\emph{Flat tangles in the surface $F$} can be defined as immersions of an oriented compact 1-manifold $T$ into the thickened surface $F\times(0,1)$ such that $\partial T=T\cap \partial F\times(0,1)$ up to homotopies fixed on the boundary $\partial T$. Combinatorially, one defines flat tangles as equivalence classes of tangle diagrams modulo isotopies, Reidemeister moves and crossing changes (Fig.~\ref{pic:cross_change}).

\begin{figure}[h]
\centering
\includegraphics[width=0.3\textwidth]{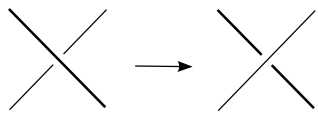}
\caption{Change of a crossing}\label{pic:cross_change}
\end{figure}

Flat tangle diagrams are obtained from tangle diagrams by forgetting under-overcrossing structure (Fig.~\ref{pic:flat_tangle}).

\begin{figure}[h]
\centering
  \includegraphics[width=0.4\textwidth]{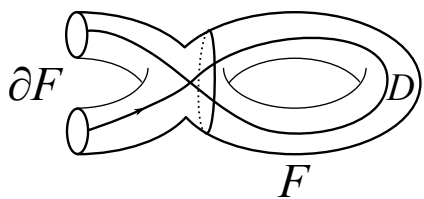}
  \caption{A flat tangle diagram}\label{pic:flat_tangle}
\end{figure}

\emph{Oriented virtual tangles} $\mathcal T$ are equivalence classes of pairs $(F,T)$ where $F$ is an oriented closed surface and $T\subset F\times(0,1)$ is an oriented tangle in the thickening of $F$, considered modulo isotopies, diffeomorphisms and stabilization/destabilization operations~\cite{CKS}. The stabilization operations are: attaching a (thickened) 2-handle distinct from $T$ to $F\times(0,1)$, sealing a boundary component distinct from $T$ with a (thickened) disk, and gluing a (thickened) handle distinct from the tangle boundary $\partial T$ to a boundary component of $F\times(0,1)$, see Fig.~\ref{pic:stabilizations}.

\begin{figure}[h]
\centering\includegraphics[width=0.5\textwidth]{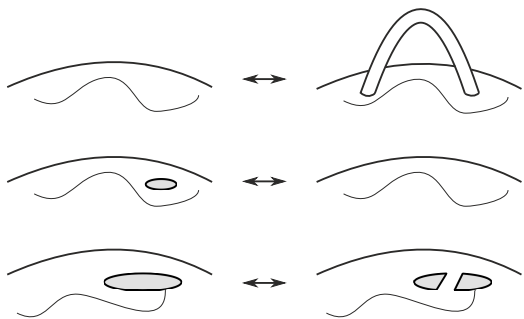}
\caption{The three stabilization operations}\label{pic:stabilizations}
\end{figure}

\emph{Virtual tangle diagrams} are pairs $(F,D)$ where $F$ is a surface and $D\subset F$ is a tangle diagram (a general projection of $T$ for some representative $(F,T)$ of $\mathcal T$). One can consider virtual tangles as the equivalence classes of diagrams up to surface isotopies (fixed on the boundary), Reidemeister moves, diffeomorphisms of pairs $(F,D)$ and (de)stabilizations.

A \emph{flat tangle diagram} is a pair $(F,D)$ where $F$ is an oriented compact connected surface and $D\subset F$ is an embedded graph with vertices of valency $4$ and vertices of valency $1$ (forming the boundary $\partial D$ of the diagram $D$) such that $D\cap\partial F=\partial D$. A \emph{flat tangle} is an equivalence class of flat link diagrams modulo isotopies of the surface, Reidemeister moves, orientation preserving diffeomorphisms of pairs $(F,D)$ and (de)stabilizations.

Another approach to virtual (flat) tangles uses plane virtual diagrams~\cite{K1}.

Fix a disk $\mathbb D$ in the plane $\R^2$.

A {\em virtual diagram} is an embedded graph $D\subset\mathbb D$ which has vertices of valency four (crossings) and one (boundary vertices). The set $\partial D$ of the vertices of valency one coincides with $D\cap\partial\mathbb D$. Each vertex of valency $4$ of the graph $D$ is marked as either {\em classical} of {\em virtual} vertex. Virtual crossings of a virtual diagram  are drawn circled (see Fig.~\ref{pic:virtual_tangle_diagram}). A diagram without virtual crossings is {\em classical}.
\begin{figure}[h]
\centering
\includegraphics[width=0.3\textwidth]{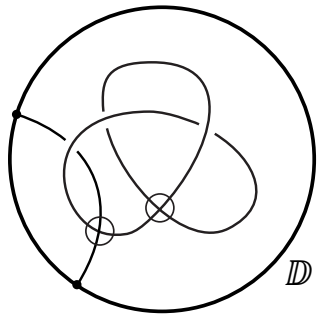}
\caption{Virtual tangle diagram with three classical and two virtual crossings}\label{pic:virtual_tangle_diagram}
\end{figure}
%

{\em Moves of virtual diagrams} include classical Reidemeister moves ($R1$, $R2$, $R3$), {\em detour moves} ($DM$) and \emph{boundary moves} (Fig.~\ref{pic:detour_move}). A detour move replaces any diagram arc, which has only virtual crossings, with a new arc, which has the same ends and contains only virtual crossings. A boundary move interchanges two neighboring boundary vertices of a virtual tangle diagram and adds a virtual crossing.  An equivalence class of virtual diagram modulo moves is called a {\em virtual tangle}.

\begin{figure}[h]
\centering
\includegraphics[width=0.8\textwidth]{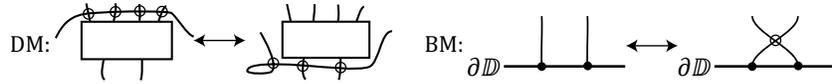}
\caption{Detour and boundary moves}\label{pic:detour_move}
\end{figure}

A {\em classical tangle} is a tangle which has at least one classical diagram.

If one admits crossing switch transformations (Fig.~\ref{pic:cross_change}) of virtual diagrams, i.e. neglects the over-undercrossing structure, one gets the theory of {\em flat tangles}.

Thirdly, oriented virtual tangles can be defined by means of Gauss diagrams~\cite{GPV}. The {\em Gauss diagram} $G=G(D)$ of a virtual tangle diagram $D$ is a chord diagram on an oriented 1-dimensional manifold $G$ which is immersed onto $D$. The chords of the chord diagram correspond to the classical crossings of $D$ (i.e. the double points of the immersion), see Fig.~\ref{pic:virtual_gauss_diagram}. The chords carry an orientation (from over-crossing to under-crossing) and the sign of the crossings, see Fig.~\ref{pic:crossing_sign}.

\begin{figure}[h]
\centering\includegraphics[width=0.3\textwidth]{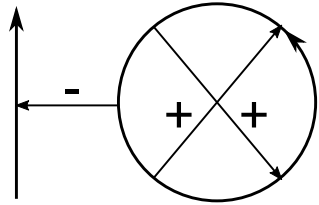}
\caption{The Gauss diagram of the virtual tangle in Fig.~\ref{pic:virtual_tangle_diagram}}\label{pic:virtual_gauss_diagram}
\end{figure}

Classical Reidemeister moves induce transformations of Gauss diagrams (see Fig.~\ref{pic:reidemeister_gauss}). {\em Virtual tangles} are  the equivalence classes of Gauss diagrams modulo the induced Reidemeister moves.

\begin{figure}[h]
\centering\includegraphics[width=0.7\textwidth]{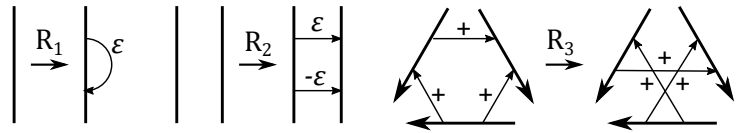}
\caption{Reidemeister moves on Gauss diagrams}\label{pic:reidemeister_gauss}
\end{figure}

\begin{theorem*}
The three definitions of virtual tangles are equivalent.
\end{theorem*}
The proof of the theorem reproduces the corresponding proof for virtual links~\cite{CKS,GPV}.

The factorization by crossing change and virtualization moves (Fig.~\ref{pic:virtualization}) leads to the theory of {\em free tangles}.

\begin{figure}[h]
\centering\includegraphics[width=0.25\textwidth]{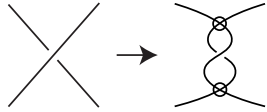}
\caption{Virtualization move}\label{pic:virtualization}
\end{figure}

Crossing change and virtualization moves induces transformations on Gauss diagrams (Fig.~\ref{pic:gauss_virtualization}). Therefore, the theory of free tangles which is obtained by factorization by these transformations, is the theory of chord diagrams without orientation and labels on the chords. The moves of free tangles are the moves in Fig.~\ref{pic:reidemeister_gauss} after one has wiped the arrowheads and the signs.

\begin{figure}[h]
\centering\includegraphics[width=0.25\textwidth]{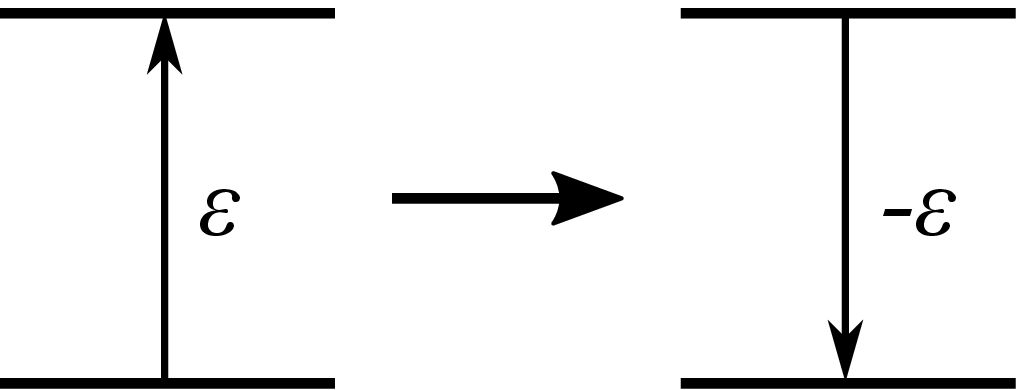}
\qquad\qquad\includegraphics[width=0.25\textwidth]{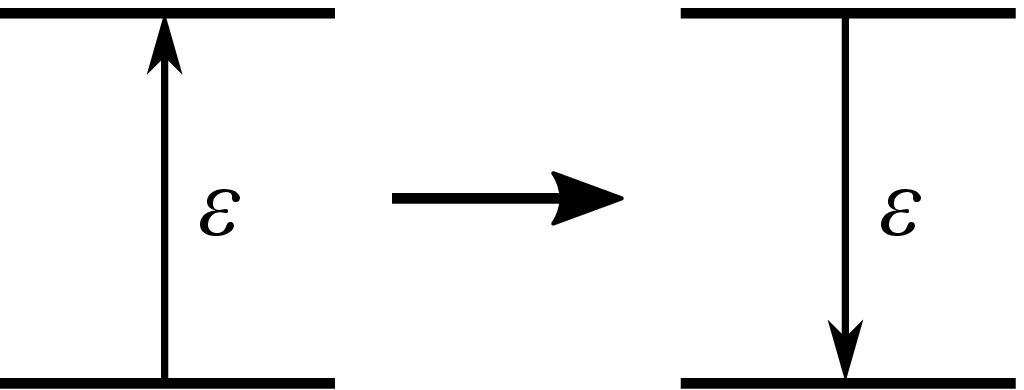}
\caption{Crossing change (left) and virtualization (right) of a Gauss diagram}\label{pic:gauss_virtualization}
\end{figure}

\subsection{Diagram categories}

Let $\mathcal{K}$ be a knot. We shall use the notion of `knot' in
one of the following situations:
 \begin{itemize}
  \item virtual knot (long knot, link or tangle);
  \item flat knot (long knot, link or tangle);
  \item free knot (long knot, link or tangle);
  \item knot (long knot, link or tangle) in a fixed surface;
  \item flat knot (long knot, link or tangle) in a fixed surface;
  \item classical knot (long knot, link or tangle).
 \end{itemize}

We always assume that the components of tangle are numbered and oriented.

\begin{definition}\label{def:knot_category}
The {\em category of diagrams $\mathfrak{K}=\mathfrak{K}(\mathcal K)$ of the knot $\mathcal{K}$}
is the small category whose objects are the diagrams of $\mathcal{K}$ and
morphisms are (formal) compositions
of {\em elementary morphisms}. By an elementary morphism we mean
 \begin{itemize}
  \item an isotopy of a diagram;
  \item an isomorphism between diagrams (in case of virtual, flat or free tangles);
  \item a Reidemeister move.
 \end{itemize}
\end{definition}

Given a morphism $f\colon D\to D'$ in $\mathfrak{K}$, generally there is no bijection between the sets of
crossings $\V(D)$ and $\V(D')$ of the diagrams because the number of crossings of a diagram may change under
Reidemeister moves. But there is a bijection for the crossings which survives the morphism: there are subsets $\bar\V(D)\subset \V(D)$ and $\bar\V(D')\subset \V(D')$ which are in an one-to-one correspondence $\bar\V(D)\to \bar\V(D')$. We denote this correspondence by $f_*$ and the subsets by $dom(f_*)=\bar\V(D)$ and $im(f_*)=\bar\V(D')$.

Usually, one works not with one knot or tangle but with knots from a certain class (e.g. virtual knots, virtual links, free knots etc.). In this case, we have to deal with diagrams of all such knots, that is work with diagram category which is the union of the categories of diagrams on the knots. We call this ``large'' categories \emph{knot theories}.

\begin{definition}\label{def:knot_theory}
Given a tangle $T$, its \emph{type} $t(T)$ is the pair $(m,n)$ where $m$ is the number of closed components in $T$ and $n$ is the number of the long components.

For a compact oriented surface $F$, the \emph{theory of tangles of type $(m,n)$ with the boundary $B$ in the surface $F$} is the diagram category
\[
\mathfrak{T}_{m,n}(F,B)=\bigsqcup_{T\colon t(T)=(m,n),\ \partial T=B}\mathfrak K(T),
\]
where $B\subset \partial F$, $|B|=2n$.

The \emph{theory of flat tangles of type $(m,n)$ with the boundary $B$ in the surface $F$} is the union
\[
\mathfrak{FT}_{m,n}(F,B)=\bigsqcup_{T\mbox{\scriptsize\ is flat}\colon t(T)=(m,n),\ \partial T=B}\mathfrak K(T).
\]
The \emph{theory of virtual tangles of type $(m,n)$} is the diagram category
\[
\mathfrak{VT}_{m,n}=\bigsqcup_{T\mbox{\scriptsize\ is a virtual tangle}\colon t(T)=(m,n)}\mathfrak K(T),
\]
the \emph{theory of flat tangles of type $(m,n)$} is the diagram category
\[
\mathfrak{FT}_{m,n}=\bigsqcup_{T\mbox{\scriptsize\ is a flat tangle}\colon t(T)=(m,n)}\mathfrak K(T),
\]
and the \emph{theory of free tangles of type $(m,n)$} is the diagram category
\[
\mathfrak{FrT}_{m,n}=\bigsqcup_{T\mbox{\scriptsize\ is a free tangle}\colon t(T)=(m,n)}\mathfrak K(T).
\]

As special cases, we get the \emph{theories of classical knots} $\mathfrak{CK}=\mathfrak{T}_{1,0}(S^2,\emptyset)$, \emph{virtual knots} $\mathfrak{VK}=\mathfrak{VT}_{1,0}$, \emph{long virtual knots} $\mathfrak{LVK}=\mathfrak{VT}_{0,1}$, \emph{virtual links with $m$ components} $\mathfrak{VL}_m=\mathfrak{VT}_{m,0}$ etc.
\end{definition}

\begin{remark}
Knot theories are connected by natural transformations. For example, we have a sequence of projections $\mathfrak{T}_{m,n}(F,B)\to \mathfrak{VT}_{m,n}\to \mathfrak{FT}_{m,n}\to \mathfrak{FrT}_{m,n}$ and  $\mathfrak{T}_{m,n}(F,B)\to \mathfrak{FT}_{m,n}(F,B)$ from (virtual) tangles to flat tangles and then to free tangles. In particular, that means invariants of free tangles can be lifted to invariants of flat tangles, and invariants of flat tangles are lifted to invariants of virtual tangles.
\end{remark}

We will use the name \emph{diagram category} both for category of diagrams of a given tangle and for knot theories defined above.

\section{Indices}\label{sect:indices}

The notion of index was axiomatized by Z. Cheng~\cite{Cheng2} in the spirit of Manturov's parity axioms~\cite{M3}. M. Xu~\cite{Xu} found the general conditions for an index to define an invariant polynomial and introduced the notion a weak chord index. See also~\cite{Cheng2,CGX,CFGMX}  for other index axiomatics and examples of chord indices.

\begin{definition}\label{def:index}
Let $\mathfrak K$ be a diagram category. An \emph{index with coefficients in a set $I$ on the diagram category $\mathfrak K$} is a map $\iota$ which assigns some value $\iota(v)\in I$ to each crossing $v\in\V(D)$ in each diagram $D\in\mathfrak K$ of the diagram category and possesses the following properties:
\begin{itemize}
\item[(I0)] for any elementary morphism $f\colon D\to D'$ in $\mathfrak K$ and any crossings $v\in dom(f_*)$ one has $\iota(v)=\iota(f_*(v))$;
\item[(I2)] $\iota(v_1)=\iota(v_2)$ for any crossings $v_1,\,v_2\in\V(D)$ to which a decreasing second Reidemeister move can be applied.
\end{itemize}

Let $S$ be a set with an involution $\ast\colon S\to S$. A \emph{signed index} with coefficients in the set $S$ on the diagram category $\mathfrak K$ is a map $\sigma\colon\bigsqcup_{D\in\mathfrak K}\V(D)\to S$ which satisfies the property (I0) and the property
\begin{itemize}
  \item[(I2+)] for any crossings $v_1,\,v_2\in\V(D)$ to which a decreasing second Reidemeister move can be applied, one has $\sigma(v_1)=\sigma(v_2)^\ast$.
\end{itemize}
\end{definition}

\begin{remark}\label{rem:index_basic}
1. For a (signed) index, the property (I0) holds for all morphisms of the diagram category $\mathfrak K$.

2. An index $\iota$ with coefficients in a set $I$ can be considered as a signed index if one takes the trivial involution on $I$: $x^*=x$ for all $x\in I$.

3. If $\sigma$ is a signed index with coefficients is a set $S$ with an involution $\ast$ then the composition $\bar\sigma=p\circ\sigma$ of $\sigma$ with the projection $p\colon S\to S/\ast$, where $S/\ast$ is the quotient set by the involution action, is an index with coefficients in the set $S/\ast$.
\end{remark}

\begin{example}\label{exa:sign_signed_index}
The sign function is a signed index with coefficients in $\Z_2=\{-1,+1\}$ for the diagram categories $\mathfrak{T}_{m,n}(F,B)$ and $\mathfrak{VT}_{m,n}$ of tangles in a surface and virtual tangles and all their diagram subcategories.
\end{example}

Usually when an index is mentioned, one means the $\Z$-valued index appeared in~\cite{Cheng,K2} or a $\Z$-valued signed index introduced in~\cite{T,Henrich,ILL,FK}. In order to differ it from the general notion of index, we will call it the \emph{Gaussian index} by analogy with the Gaussian parity~\cite{IMN1}.

\begin{example}[Gaussian index]\label{exa:classic_index}
Let $K$ be a knot in a surface $F$ or a virtual knot. For any diagram $D$ of $K$ and any crossing $v\in\V(D)$, the oriented smoothing at $v$ splits the component $D_i$ into two the \emph{left half} $D^l_v$ and the \emph{right half} $D^r_v$  (Fig.~\ref{pic:knot_halves}). We define also the \emph{signed halves} $D^\pm_v$ by the formula $D^{sgn(v)}_v=D^l_v$ and $D^{-sgn(v)}_v=D^r_v$.

\begin{figure}[h]
\centering\includegraphics[width=0.4\textwidth]{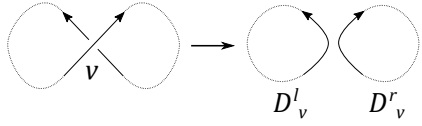}
\caption{The left and the right halves of the diagram}\label{pic:knot_halves}
\end{figure}

Then define the \emph{Gaussian index} of the crossing as the intersection number of two cycles: $Ind(v)=D\cdot D^+_v\in\Z$. Then the map $Ind$ obeys the conditions (I0) and (I2). It is an index with coefficients in $\Z$ on diagram categories of  virtual knots and knots in a fixed surface.

The index $n$ defined by V. Turaev~\cite{T} can be obtained by the formula $n(v)=D\cdot D^l_v$,. In fact, it is a signed index with coefficients in $\Z$ on diagram categories of virtual and flat knots and (flat) knots in a fixed surface.

For virtual knots, the indices $Ind$ and $n$ are connected by the formula $Ind(v)=sgn(v)\cdot n(v)$.
\end{example}

In the presence of the sign function there is a correspondence between indices and signed indices as the previous example shows. Let us describe it in a more general setting.

\begin{proposition}\label{prop:sign_and_index}
Let $\mathfrak K$ be a diagram category $\mathfrak T_{m,n}(F,B)$, $\mathfrak{VT}_{m,n}$ or any subcategory of those.
\begin{enumerate}
  \item If $\sigma$ is a signed index with coefficients in a set $S$ with an involution $\ast$. Then the map $\hat\sigma$ such that $\hat\sigma(v)=\sigma(v)^{sgn(v)}$, where $v\in\V(D)$, $D\in\mathfrak K$, and $x^{+1}=x$ and $x^{-1}=x^\ast$ for $x\in S$, is an index with coefficients in the set $S$ on the diagram category $\mathfrak K$.
  \item  If $\iota$ is an index with coefficients in a set $S$ with an involution $\ast$. Then the map $\hat\iota$ such that $\hat\iota(v)=\iota(v)^{sgn(v)}$, $v\in\V(D)$, $D\in\mathfrak K$, is a signed index with coefficients in the set $S$ on the diagram category $\mathfrak K$.
  \item If $\iota$ is an index with coefficients in a set $I$ on $\mathfrak K$. Then the map $\tilde\iota$ such that $\tilde\iota(v)=(\iota(v),sgn(v))\in I\times\{-1,+1\}$ is a signed index on $\mathfrak K$ with coefficients in the set $I\times\{-1,+1\}$ with the involution $(x,\epsilon)^\ast=(x,-\epsilon)$, $x\in I$, $\epsilon\in\{-1,+1\}$.
\end{enumerate}
\end{proposition}

Let us describe how a (signed) index relates first Reidemeister move. A first Reidemeister move can create a loop crossing of one of the four types in Fig.~\ref{pic:loop_types}.

\begin{figure}[h]
\centering\includegraphics[width=0.35\textwidth]{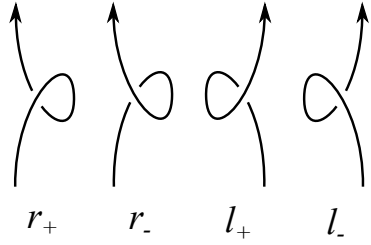}
\caption{Types of loops}\label{pic:loop_types}
\end{figure}

\begin{proposition}\label{prop:index_r1}
Let $\mathfrak K=\mathfrak K(\mathcal T)$ be the category of diagrams of a tangle $\mathcal T=T_1\cup\cdots\cup T_l$ with $l$ component, $\iota$ be an index on $\mathfrak K$ with coefficients in a set $I$, and $\sigma$ be a signed index on $\mathfrak K$ with coefficients in a set with involution $S$. Then there exist values $\iota_i^{l\pm}, \iota_i^{r\pm}\in I$ and  $\sigma_i^{l\pm}, \sigma_i^{r\pm}\in S$, $i=1,\dots,l$, such that

\begin{enumerate}
  \item any loop crossing $v$ of type $l+$ (corr., $l-$, $r+$, $r-$) in a component $T_i$ has the index value $\iota_i^{l+}$ (corr., $\iota_i^{l-}, \iota_i^{r+}, \iota_i^{r-}$) and the signed index value $\sigma_i^{l+}$ (corr., $\sigma_i^{l-}, \sigma_i^{r+}, \sigma_i^{r-}$);
  \item $\iota_i^{l+}=\iota_i^{r-}$, $\iota_i^{l-}=\iota_i^{r+}$, and  $\sigma_i^{r-}= (\sigma_i^{l+})^\ast$, $\sigma_i^{r+}= (\sigma_i^{l-})^\ast$, $i=1,\dots,l$.
\end{enumerate}
\end{proposition}

\begin{proof}
Due to Remark~\ref{rem:index_basic} it is enough to give a proof for the signed index.

Let $v\in\V(D_1)$ be any loop index of type $l+$ in a component $T_i$, and $w\in\V(D_2)$ be a loop crossing of type $r-$ in $T_i$, $D_1,D_2\in\mathfrak K$. We show that $\sigma(w)=\sigma(v)^\ast$.

Since $D_1$ and $D_2$ are diagrams of one tangle $\mathcal T$, there is a morphism $f\colon D_2\to D_1$. We modify the morphism $f$ as follows: we add a double $w'$ of the crossing $w$ in the same component of the diagram $D_2$ and use the double instead of $w$ in the Reidemeister moves of the morphism $f$, and keep the loop with the crossing $w$ somewhere on the diagrams. Thus, we get a sequence of diagrams $D_2\stackrel{R_1}{\rightarrow}D_2'\stackrel{f'}{\rightarrow} D'_1$ where $D_2'$ is the diagram with the double added, $f'$ is the morphism $f$ with added loop $w$, and $D'_1$ differs from $D_1$ by a first Reidemeister move which adds the loop crossing $w''=f'_*(w)$.

We can move the crossing $f'_*(w)$ to the crossing $v$ and assume that these crossings are neighboring in the diagram  $D_1'$. Then move the crossing $w''$ along the component as shown in Fig.~\ref{pic:index_r1}. Then by the property $(I2+)$  $\sigma(w'')=\sigma(v)^\ast$. By the property $(I0)$ $\sigma(w)=\sigma(f'_*(w))=\sigma(w'')=\sigma(v)^\ast$.

\begin{figure}[h]
\centering\includegraphics[width=0.5\textwidth]{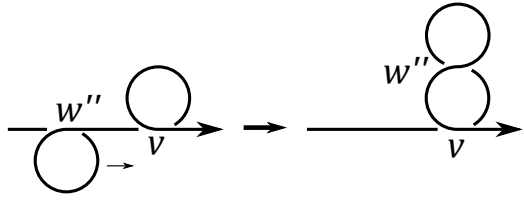}
\caption{Pairing loops}\label{pic:index_r1}
\end{figure}

Let $v\in\V(D)$ and $v'\in\V(D')$ be two loop crossings of the type $l+$. Create in some diagram $D''\in\mathfrak K$ a loop crossing $w$ of type $r-$ in the $i$-th component. Then $\sigma(v')=\sigma(w)^\ast=(\sigma(v)^\ast)^\ast=\sigma(v)$. Thus, all the loop crossings of type $l+$ in the component $T_i$ have the same signed index value $\sigma_i^{l+}\in S$. The proof for other types is analogous.

The equality $\sigma_i^{r-}=(\sigma_i^{l+})^\ast$ was proved above. Analogously, $\sigma_i^{r+}=(\sigma_i^{l-})^\ast$.
\end{proof}

\begin{remark}\label{rem:index_r1}
\begin{enumerate}
  \item The values $\iota_i^{l\pm}, \iota_i^{r\pm}\in I$ and $\sigma_i^{l\pm}, \sigma_i^{r\pm}\in S$ can coincide;
  \item if $\mathcal T$ is a flat tangle and $\iota$ is an index, and $\sigma$ is a signed index on the category $\mathfrak K(\mathcal T)$ of diagrams of $\mathcal K$, then the crossings don't have signs, and the loop crossings in one component of the tangle have one index value $\iota_i=\iota_i^{l\pm}=\iota_i^{r\pm}$, and two signed index values $\sigma_i^l$ and $\sigma_i^r=(\sigma_i^l)^\ast$;
  \item if $\mathcal K$ is a free tangle then all the loop crossings on the same component have the same index or signed index value.
  \item if one considers a knot theory $\mathfrak K$ which consists of diagram categories $\mathfrak K(T)$ of different tangles, the loop indices can be individual for each tangle $T$: $\iota_i^{l\pm}(T), \iota_i^{r\pm}(T)$.

      Some indices may have one loop index value for different tangles. For example, for the $\Z$-valued indices $Ind$ and $n$, the loop index values are all equal to zero.
\end{enumerate}
\end{remark}

Thus, the first Reidemeister move can be localized and isolated by neglecting the index values $\iota^{l\pm},\iota^{r\pm}$ ($\sigma^{l\pm}, \sigma^{r\pm}$). We can use this fact and define a tangle invariant (the index polynomial).

\subsection{Index polynomial}

\begin{definition}\label{def:linking_invariant}
Let $\mathfrak K$ be a diagram category and $\sigma$ be a signed index on $\mathfrak K$ with coefficients in a set $S$. Let $L\subset S$ be the set of loop index values, i.e. $L=\bigcup_{i}\{\sigma_i^{l\pm},\sigma_i^{r\pm}\}$ if $\mathfrak K$ is the category of diagrams of a tangle $T=\bigcup_i T_i$, or $L=\bigcup_{T\in{\mathfrak K},\,i}\{\sigma_i^{l\pm}(T),\sigma_i^{r\pm}(T)\}$ if $\mathfrak K$ is a knot theory. Note that $L^\ast=L$.

Consider the abelian group $A_\sigma=\Z[S\setminus L]/\langle 1\cdot x+1\cdot x^\ast \mid x\in S\setminus L\rangle$.

Let $T$ be a tangle presented by a diagram $D\in\mathfrak K$. The \emph{linking invariant} (or \emph{(odd) index polynomial}) $lk_\sigma(T)$ of the tangle $T$ is the element
\begin{equation}\label{eq:signed_index_linking_invariant}
  lk_\sigma(T)=\sum_{v\in\V(D)\,\colon\,\sigma(v)\not\in L}1\cdot\sigma(v)\in A_\sigma.
\end{equation}

For an index $\iota$ on $\mathfrak K$ with coefficients in a set $I$ we define the \emph{linking invariant} $lk_\iota(T)$ as $lk_{\tilde\iota}(T)$ if the category $\mathfrak K$ has the sign function, and as $lk_\iota(T)$ if $\mathfrak K$ is a diagram category of flat or free tangles. In the latter case we consider the trivial involution on the set $I$.
\end{definition}

\begin{proposition}\label{prop:linking_invariant}
The index polynomial $lk_\sigma$ is an invariant of tangles in the diagram category $\mathfrak K$, i.e. the value $lk_\sigma(T)$ does not depend on the choice of a diagram $D$ representing the tangle $T$.
\end{proposition}

\begin{proof}
It is enough to prove that for any elementary morphism  $f\colon D_1\to D_2$ in $\mathfrak K$ the values of $lk_\sigma$ for $D_1$ and $D_2$ coincide. When $f$ does not change the number of crossing, this follows from the property (I0). If $f$ is a second Reidemeister move then the equality follows from the property (I2). We secure the invariance under first Reidemeister moves by excluding the loop index values.
\end{proof}

\begin{remark}\label{rem:linking_invariant}
1. The group $A_\sigma$ is a direct sum of the cyclic groups $\Z$ (one summand for each pair $\{x,x^\ast\}$ such that $x\in S\setminus L$ and $x\ne x^\ast$) and $\Z_2$ (one summand for each self-dual element $x=x^\ast\in S\setminus L$).

For an index $\iota$ with coefficients in a set $I$, the group $A_\iota=A_{\tilde\iota}$ is isomorphic to $\Z[I\setminus L]$ in the presence of the sign function on the crossings (one maps $(x,\epsilon)$ to $\epsilon\cdot x$), and $A_\iota=\Z_2[I\setminus L]$ for diagram categories of flat or free tangles.

2. Let $\iota$ be an index on a diagram category $\mathfrak K$ with coefficients in a set $I$, and the crossings of the diagrams in $\mathfrak K$ have signs. Using the isomorphism  $A_\iota\simeq\Z[I\setminus L]$ we can write the formula the odd index polynomial of a tangle $T$ as follows
\begin{equation}\label{eq:index_linking_polynomial}
  lk_\iota(T)=\sum_{v\in\V(D)\,\colon\,\iota(v)\not\in L}sgn(x)\cdot\iota(v)\in\Z[I\setminus L]
\end{equation}
where $D\in\mathfrak K$ is a diagram of the tangle $T$.

The formula~\eqref{eq:index_linking_polynomial} can be considered as a generalization of the well-known formula for the linking number. This justifies the notation $lk$ which we use for the odd index polynomial.
\end{remark}

\begin{example}\label{exa:classical_index_polynomial}
1. Let $\mathfrak{VK}$ be the theory of virtual knots. Let $Ind$ be the index with coefficients in $\Z$ defined in Example~\ref{exa:classic_index}. The set of loop index values is $L=\{0\}$. Then the linking invariant of a virtual knot $K$ given by a diagram $D\in\mathfrak{VK}$
\[
lk_{Ind}(K)=\sum_{v\in\V(D)\,\colon\, Ind(v)\ne 0}sgn(v)\cdot Ind(v)\in\Z[\Z\setminus\{0\}]
\]
is (up to isomorphism and the multiplier $t$) the odd index polynomial $f(t)$ defined in~\cite{Cheng}.

2. Let $\mathfrak{FK}$ be the theory of flat knots. Let $n$ be the index with coefficients in $\Z$ defined in Example~\ref{exa:classic_index}. The set of loop index values is $L=\{0\}$. The group $A_n$ is isomorphic to the group $\Z[\mathbb N]$  by the formula $x\mapsto sgn(x)\cdot |x|$. Using this isomorphism,  we write the linking invariant of a flat knot $K$ with a diagram $D\in\mathfrak{FK}$ as
\[
lk_{n}(K)=\sum_{v\in\V(D)\,\colon\, n(v)\ne 0}sgn(n(v))\cdot |n(v)|\in\Z[\mathbb N].
\]
The linking invariant $lk_{n}(K)$ coincides with the polynomial invariant $u(K)(t)\in\Z[t]$ defined by Turaev~\cite{T}.
\end{example}

\subsection{Universal index}

Let $\mathfrak K$ be a diagram category. There can be many (signed) indices on $\mathfrak K$. We are interested in the (signed) index which carries the maximal information about the crossings.

\begin{definition}\label{def:universal_index}
Let $\mathfrak K$ be a diagram category. An index $\iota^u$ with coefficients in a set $I^u$ on the diagram category $\mathfrak K$  is called the \emph{universal index} on $\mathfrak K$ if for any index $\iota$ on $\mathfrak K$ with coefficients in a set $I$ there is a unique map $\psi\colon I^u\to I$ such that $\iota=\psi\circ\iota^u$.

A signed index $\sigma^u$ on the diagram category $\mathfrak K$ with coefficients in a set $S^u$ with an involution $\ast$ is \emph{universal} if for any signed index $\sigma$ on $\mathfrak K$ with coefficients in a set with involution $S$ there is a unique map $\psi\colon S^u\to S$ such that  $\psi(x)^\ast=\psi(x^\ast)$ for all $x\in S^u$ and $\sigma=\psi\circ\sigma^u$.
\end{definition}

\begin{remark}
The universal index for the diagram categories $\mathfrak K(T)$ of tangles and flat tangles in a surface $F$ was described in~\cite{Nct}.
\end{remark}

\begin{definition}\label{def:reduced_index}
A (signed) index $\sigma$ on a diagram category $\mathfrak K$ with coefficients in a set $S$ is called \emph{reduced} if $S=im(\sigma)$ where
\[
im(\sigma)=\{\sigma(v)\mid v\in\V(D),\ D\in\mathfrak K\}.
\]
\end{definition}

\begin{proposition}\label{prop:basic_universal_index}
1. The universal (signed) index on a diagram category is reduced.

2. Let $\sigma^u$ be the universal signed index on a diagram category $\mathfrak K$. Then $\bar{\sigma}^u$ (see Remark~\ref{rem:index_basic}) is the universal index on $\mathfrak K$.

3. If a diagram category $\mathfrak K$ has the sign function (i.e. it is a subcategory of $\mathfrak T_{m,n}(F,B)$ or $\mathfrak{VT}_{m,n}$) and $\iota^u$ is the universal index on $\mathfrak K$ then $\tilde\iota^u$ is the universal signed index on $\mathfrak K$.
\end{proposition}

\begin{proof}
1. Let $\mathfrak K$ be a diagram category and $\sigma^u$ with coefficients in a set $S^u$ with an involution $\ast$ be the universal signed index on $\mathfrak K$. Let $S'=im(\sigma^u)\subset S$. Then the co-restriction $\sigma'$ of the index $\sigma^u$ to the coefficient set $S'$ is a signed index on $\mathfrak K$. By universality there exists a unique $\ast$-map $\psi\colon S^u\to S'$ such that $\sigma'=\psi\circ\sigma^u$. Then the composition $\psi'$ of maps $S^u\stackrel{\psi}\rightarrow S'\hookrightarrow S^u$ satisfies the condition $\psi'\circ\sigma^u=\sigma^u$. Then by universality, the map $\psi'$ must coincide with the identity map $id_{S^u}$. Thus, $S'=S^u$.

2. Let $\iota^u$ with coefficients in a set $I^u$ be the universal index on $\mathfrak K$. By universality of $\iota^u$ there exists a map $\psi\colon I^u\to S^u/\ast$ such that $\bar\sigma^u=\psi\circ\iota^u$. On the other hand, by universality of $\sigma^u$ there exists a map $\phi\colon S^u\to I^u$ such that $\phi(x^\ast)=\phi(x)^\ast=\phi(x)$ for all $x\in S^u$ and $\iota^u=\phi\circ\sigma^u$. The map $\phi$ induces a well-defined map $\bar\phi\colon S^u/\ast\to I^u$. Then the maps $\psi$ and $\bar\phi$ define an isomorphism of the indices $\iota^u$ and $\bar\sigma^u$.

3. By universality of $\sigma^u$ there is a map $\psi$ from $S^u$ to $I^u\times\{-1,+1\}$. This map is injective. Indeed, if $\psi(x)=\psi(y)$ then take crossings $v$ and $w$ such that $\sigma^u(v)=x$ and $\sigma^u(w)=y$. Since $\iota^u(v)=\iota^u(w)$, either $\sigma^u(v)=\sigma^u(w)$ or $\sigma^u(v)=\sigma^u(w)^\ast$. In the latter case the crossings $v$ and $w$ would have different signs. Thus, $\sigma^u(v)=\sigma^u(w)$. Then $\psi$ is a bijection.
\end{proof}


\section{Traits and indices}\label{sect:traits}

The aim of this section is to prove the central result of the paper that one can weaken the definition of the signed universal index by removing the condition (I2+). Let us give the necessary definition.

\begin{definition}\label{def:trait}
Let $\mathfrak K$ be a diagram category. A \emph{trait with coefficients in a set $\Theta$ on the diagram category $\mathfrak K$} is a map $\theta\colon \bigsqcup_{D\in\mathfrak K}\V(D)\to\Theta$ which satisfies the property (I0): for any elementary morphism $f\colon D\to D'$ in $\mathfrak K$ and any crossings $v\in dom(f_*)$ one has $\theta(v)=\theta(f_*(v))$.
\end{definition}

By definition any index or signed index is a trait.

\subsection{Universal trait}

By analogy with the (signed) index case, one can define the universal trait.

\begin{definition}\label{def:universal_trait}
Let $\mathfrak K$ be a diagram category. An trait $\theta^u$ with coefficients in a set $\Theta^u$ on the diagram category $\mathfrak K$  is called the \emph{universal trait} on $\mathfrak K$ if for any trait $\theta$ on $\mathfrak K$ with coefficients in a set $\Theta$ there is a unique map $\psi\colon \Theta^u\to \Theta$ such that $\theta=\psi\circ\theta^u$.
\end{definition}

It is clear that the universal signed index is a trait. It appears that some kind of inverse statement is true.

\begin{theorem}\label{thm:main_theorem}
Let $\theta^u$ be the universal trait with coefficients in a set $\Theta^u$ on a diagram category $\mathfrak K$. Then $\theta^u$ is the universal signed index on $\mathfrak K$ (with respect to some involution $\ast$ on the set $\Theta^u$).
\end{theorem}

Let us give a simple but implicit and tautological description of the universal trait.

\begin{definition}\label{def:based_diagram_category}
Let $\mathfrak K$ be a diagram category. The \emph{based diagram category $\mathfrak K_b$ associated with $\mathfrak K$} is the category with objects $(D,v)$ where $D\in\mathfrak K$ and $v\in\V(D)$ and morphisms $f\colon (D,v)\to (D',v')$ such that $f\colon D\to D'$ is a morphism in $\mathfrak K$, $v\in dom(f_*)$ and $v'=f_*(v)$.
\end{definition}

Informally, the associated based category is category of diagrams with a distinguished crossing and morphisms (isotopies, isomorphisms and Reidemeister moves) which don't erase the distinguished crossing.

Let $\mathfrak K$ be a diagram category and $\mathfrak K_b$ be the associated based diagram category. Denote the set of objects of the based category by $\V(\mathfrak K)$. These is an equivalence relation on the set $\V(\mathfrak K)$: $(D,v)\sim (D',v')$ iff there is a morphism $f\colon (D,v)\to (D',v')$. Denote the set of equivalence classes by $\overline{\V(\mathfrak K)}$. The following proposition is evident.

\begin{proposition}\label{prop:universal_trait_description}
The map $\theta^u$ such that $\theta^u(v)=[(D,v)]\in\overline{\V(\mathfrak K)}$, $v\in\V(D)$, $D\in\mathfrak K$, is the universal trait on the diagram category $\mathfrak K$.
\end{proposition}

Let us give an analogous interpretation of the universal (signed) index.

\subsection{The graph $\mathcal G(\mathfrak K)$}

Let $\mathfrak K$ be a diagram category and $\mathfrak K_b$ be the associated based diagram category. The set $\V(\mathfrak K)$ can be identified with the set $\bigsqcup_{D\in\mathfrak K}\V(D)$ of crossing of all the diagrams in $\mathfrak K$. Let $\mathcal G(\mathfrak K)$ be the graph with the set of vertex $V(\mathcal G(\mathfrak K))=\V(\mathfrak K)$ and the set of edges $E(\mathcal G(\mathfrak K))$ which consists of edges of two types:
\begin{itemize}
  \item I0-edge: $v f_*(v)$ where $f\colon D\to D'$ is an elementary morphism in $\mathfrak K$ and $v$ is a crossing in $dom(f_*)$;
  \item I2-edge: $v_1v_2$ where $v_1$ and $v_2$ are crossings in a diagram $D\in\mathfrak K$ to which a decreasing second Reidemeister move can be applied.
\end{itemize}
Let $\mathcal G_0(\mathfrak K)$ be the subgraph in $\mathcal G(\mathfrak K)$ obtained by removing all  I2-edges from $\mathcal G(\mathfrak K)$.

The proposition below follows from the definitions.
\begin{proposition}\label{prop:universal_index_graph}
Let $\mathfrak K$ be a diagram category.
\begin{enumerate}
  \item The universal index coefficient set $I^u$ is identified with the set of connected components of the graph $\mathcal G(\mathfrak K)$. The universal index $\iota^u$ maps a crossing to the correspondent connected component.
  \item The universal trait coefficient set $\Theta^u$ is identified with the set of connected components of the graph $\mathcal G_0(\mathfrak K)$. The universal index $\theta^u$ maps a crossing to the correspondent connected component.
  \item Let $v\in\V(D)$ and $v'\in\V(D')$, $D,D'\in\mathfrak K$, be two crossing. Then $\sigma^u(v)=\sigma^u(v')$ if and only if $v$ and $w$ can be connected in $\mathcal G(\mathfrak K)$ by a path which contains an even number of I2-edges, and $\sigma^u(v)^\ast=\sigma^u(v')$ if and only if $v$ and $w$ can be connected by a path which contains an odd number of I2-edges.
\end{enumerate}
\end{proposition}
Note that two based diagrams $(D,v)$ and $(D',v')$ are equivalent in $\mathfrak K_b$ iff the crossings $v$ and $v'$ belong to one connected component of the graph $\mathcal G_0(\mathfrak K)$.

Thus, if  we want to show that the universal trait $\theta^u$ coincides with the universal signed index $\sigma^u$ then we have to prove that any path in $\mathcal G(\mathfrak K)$ which contains an even number of I2-edges, can be replaced with a path in $\mathcal G_0(\mathfrak K)$ with the same ends. For this, we use a combinatorial trick called \emph{wrapping swap}.

\subsection{Wrapping}

\begin{definition}\label{def:crossing_wrapping}
Let $\mathfrak K$ be a diagram category and $\mathfrak K_b$ be the associated based diagram category. Let $(D,v)$ be a based diagram in $\mathfrak K_b$ and $n\in\Z$. Define the \emph{wrapped based diagram} $(D,v,n)$ as the diagram $D'$ obtained from $D$ by the rotating the overcrossing of $v$ counterclockwise by $n$ half-turn as shown in Fig.~\ref{pic:wrappings}. The central crossing $v'$ of the wrapping in $D'$ is the distinguished crossing of the based diagram $(D,v,n)=(D',v')$.

\begin{figure}[h]
\centering\includegraphics[width=0.5\textwidth]{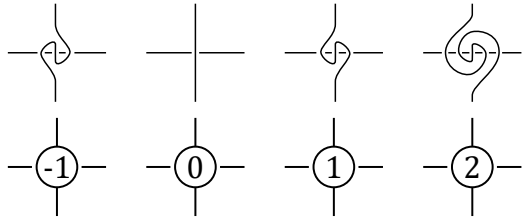}
\caption{Wrappings of a crossing}\label{pic:wrappings}
\end{figure}

We denote the wrapped crossing by a disk with the wrapping number $n$ inside.
\end{definition}

Note that the original diagram $D$ and the wrapped diagram $D'$ are equivalent in $\mathfrak K$.

\begin{proposition}\label{prop:wrapping_properties}
Let $\mathfrak K$ be a diagram category and $\mathfrak K_b$ be the associated based diagram category.
\begin{enumerate}
  \item Let $f\colon (D,v)\to(D',v')$ is a morphism in $\mathfrak K_b$. Then for any $n\in\Z$ there is a morphism $f_n\colon (D,v,n)\to(D',v',n)$ in $\mathfrak K_b$ of the wrapped diagrams.
  \item Let $(D,v)\in\mathfrak K_b$ be a based diagram. Then for any $n\in\Z$ the wrapped diagrams $(D,v,n)$ and $(D,v,n\pm2)$ are equivalent in $\mathfrak K_b$.
  \item(wrapping swap) Let $D\in\mathfrak K$ be a diagram and $v_1,v_2\in\V(D)$ be two crossings to which a second Reidemeister move can be applied. Then for any $n\in\Z$ the wrapped diagrams $(D,v_1,n)$ and $(D,v_2,n\pm 1)$ are equivalent in $\mathfrak K_b$ (Fig.~\ref{pic:wrapping_swap}).
\end{enumerate}
\begin{figure}[ht]
\centering\includegraphics[width=0.5\textwidth]{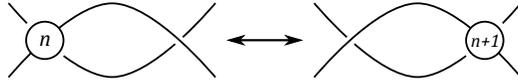}
\caption{Wrapping swap}\label{pic:wrapping_swap}
\end{figure}
\end{proposition}

\begin{proof}
  1. It is enough to consider an elementary morphism $f$. If $f$ is not a Reidemeister move or the crossing $v$ does not take part in the move $f$ then the morphism $f$ induces an elementary morphism $f_n\colon (D,v,n)\to(D',v',n)$. If $v$ participates in $f$ then $f$ is a third Reidemeister move. In this case we can perform the analogous move on the wrapped diagrams (see Fig.~\ref{pic:wrapping_r3})
\begin{figure}[ht]
\centering\includegraphics[width=0.5\textwidth]{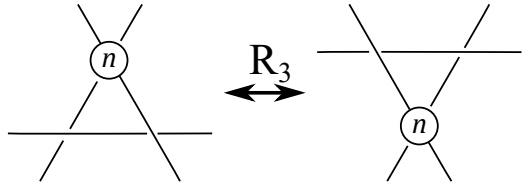}
\caption{Third Reidemeister move with a wrapped crossing}\label{pic:wrapping_r3}
\end{figure}

2.  The proof for the second statement is given in Fig.~\ref{pic:wrapping_reduction}. We move one arc below the wrapped overcrossing, and the other move above the wrapped overcrossing.
\begin{figure}[ht]
\centering\includegraphics[width=0.35\textwidth]{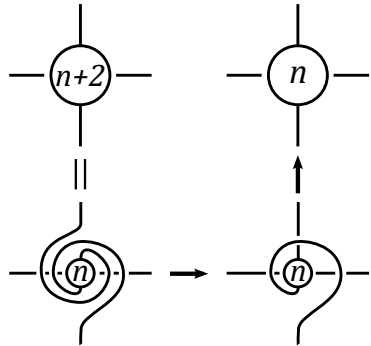}
\caption{Wrapping reduction}\label{pic:wrapping_reduction}
\end{figure}

3. The proof for the wrapping swap is given in Fig.~\ref{pic:wrapping_swap_proof}. It is realized by one second Reidemeister move.
\begin{figure}[ht]
\centering\includegraphics[width=0.9\textwidth]{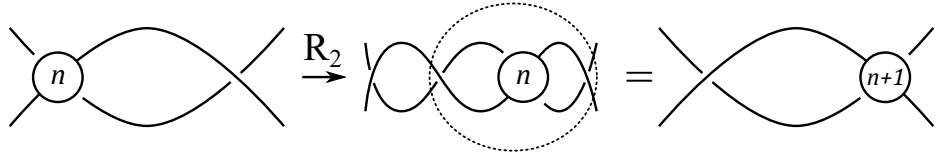}
\caption{Proof for the wrapping swap}\label{pic:wrapping_swap_proof}
\end{figure}
\end{proof}

\begin{remark}\label{rem:wrapping_free_knots}
If $\mathfrak K$ is a diagram category of free tangles then due to virtualization there can be different types of wrapping (see Fig.~\ref{pic:wrapping_free}).
\begin{figure}[ht]
\centering\includegraphics[width=0.4\textwidth]{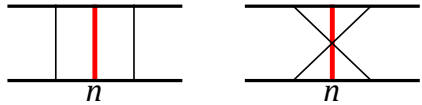}
\caption{Wrappings with wrapping number $n+1$. The central chord has wrapping of order $n$}\label{pic:wrapping_free}
\end{figure}
In this case we start proof with the second statement of Proposition~\ref{prop:wrapping_properties}. There two types of wrapping of order 2 (up to orientation of tangle components), see Fig.~\ref{pic:wrapping_free_order2}. In the first case we can reduce the wrapping to the single chord $v$ with two second Reidemeister moves, the second case corresponds to the configuration in Fig.~\ref{pic:wrapping_reduction} (with $n=0$).
\begin{figure}[ht]
\centering\includegraphics[width=0.4\textwidth]{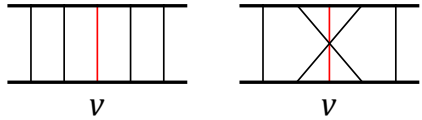}
\caption{Wrappings of order $2$}\label{pic:wrapping_free_order2}
\end{figure}
Thus, we can consider only wrappings of order $0$ or $1$. Then the other statements in Proposition~\ref{prop:wrapping_properties} are proved by considering possible configurations.
\end{remark}

Now we are ready to prove the main theorem.

\subsection{Proof of the main theorem}

By universality of the trait $\theta^u$ there is a map $\psi\colon \Theta^u\to S^u$ such that $\sigma^u=\psi\circ\theta^u$. We will prove that $\psi$ is an injection. Since the universal index and trait are reduced, this would imply that $\psi$ is a bijection and $\sigma^u$ and $\theta^u$ coincide.

Let $v\in\V(D)$ and $v'\in\V(D')$, $D,D'\in\mathfrak K$, be two crossing such that $\sigma^u(v)=\sigma^u(v')$. By Proposition~\ref{prop:universal_index_graph} that means that there exists a path $\gamma=v_1v_2\dots v_l$ in $\mathcal G(\mathfrak K)$ such that $v_1=v$, $v_l=v'$ and $\gamma$ contains even number of I2-edges. Let $D=D_1, D_2,\dots, D_l=D'$ be the diagrams the crossings $v_1,\dots, v_l$ belong to. We construct a correspondent sequence of wrapped based diagrams $(D_i,v_i,n_i)$ using induction.

We start with the diagram $(D,v,0)=(D,v)$. Assume we have defined the wrapped diagram $(D_i,v_i,n_i)$, $1\le i<l$. If the edge $v_iv_{i+1}$ is an I0-edge then there is an elementary morphism $f\colon D_i\to D_{i+1}$ such that $f_*(v_i)=v_{i+1}$. Then we take $n_{i+1}=n_i$, so there is a morphism $f_n\colon (D_i,v_i,n_i)\to (D_{i+1},v_{i+1},n_{i+1})$ in the based diagram category $\mathfrak K_b$.

If the edge $v_iv_{i+1}$ has type I2 then we set $n_{i+1}=n_i\pm 1$ and consider the wrapping swap $f_{ws}\colon (D_i,v_i, n_i)\to (D_{i+1},v_{i+1},n_{i+1})$.

Thus, we get a sequence of based diagrams $(D_i,v_i,n_i)$ and morphisms between them in the based diagram category $\mathfrak K_b$. That is all diagrams $(D_i,v_i,n_i)$, $i=1,\dots,l$ are equivalent in $\mathfrak K_b$. The sequence finishes with the diagram $(D_l,v_l,n_l)=(D',v',n_l)$. Since only I2-edges change the wrapping number and the number of such edges in $\gamma$ is even, the number $n_l$ is even. By Proposition~\ref{prop:wrapping_properties} the diagram $(D',v',n_l)$ is equivalent to the based diagram $(D',v',0)=(D',v')$.

Thus, the based diagrams $(D,v)$ and $(D',v')$ are equivalent in $\mathfrak K_b$, so by Proposition~\ref{prop:universal_trait_description}, $\theta^u(v)=\theta^u(v')$. Theorem~\ref{thm:main_theorem} is proved.

\begin{remark}\label{rem:universal_trait_involution}
Since the universal trait and the universal signed index coincide, there is an involution on the universal trait coefficient set $\Theta^u$. In terms of Proposition~\ref{prop:universal_trait_description} we can define it as follows: for any $[(D,v)]\in\overline{\V(\mathfrak K)}$ its dual element is the wrapped diagram $[D,v]^\ast=[(D,v,1)]$.
\end{remark}

\subsection{Corollaries}

Let us formulate several consequences of Theorem~\ref{thm:main_theorem}.

\begin{corollary}\label{cor:universal_signed_index_to_trait}
Let $\mathfrak K$ be a diagram category and $\theta$ be a trait on $\mathfrak K$ with coefficients in a set $\Theta$. Then there is a unique map $\psi\colon S^u\to \Theta$ from the coefficient set $S^u$ of the universal signed index $\sigma^u$ on $\mathfrak K$ such that $\theta=\phi\circ\sigma^u$.
\end{corollary}

\begin{remark}
Not all traits are signed indices because the map $\psi$ can be incompatible with the involution.

For example, let $\mathfrak K=\mathfrak{FK}$ be the theory of flat knots, $n$ be the Gaussian signed index and a trait $\theta$ with coefficients in $\Z_2$ be defined by the formula
\[
\theta(v)=\left\{\begin{array}{cl}
                   1, & n(v)\in\{-1,1,2\}, \\
                   0, & \mbox{otherwise}.
                 \end{array}\right.
\]
The trait $\theta$ is neither an index nor a signed index.
\end{remark}

\begin{theorem}\label{thm:universal_index_description}
Let $\mathfrak K$ be a diagram category. Let $D\in\mathfrak K$ and $v\in\V(D)$ be a crossing of the diagram $D$.

\begin{enumerate}
  \item The based diagram $(D,v)$ considered modulo morphisms of $\mathfrak K$ which do not eliminate chosen crossing $v$ is the universal signed index $\sigma^u(v)$ of the crossing $v$.
  \item The based diagram $(D,v)$ considered modulo morphisms of $\mathfrak K$ which do not eliminate chosen crossing $v$, and wrappings of the chosen crossing, is the universal index $\iota^u(v)$ of the crossing $v$.
\end{enumerate}
\end{theorem}

\begin{remark}\label{rem:universal_index_flat_knots}
For the theory of flat knots $\mathfrak{FK}$, the universal index of a crossing $v$ in a flat knot diagram $K$ coincides with the equivalence class $[\tilde K^v_{glue}]$ of a singular flat knot with one singular point introduced by A. Henrich~\cite{Henrich}. The universal signed index of a crossing $v$ in a flat knot diagram $K$ coincides with homotopy class $[K_v^+]$ of signed singular strings introduced by P. Cahn~\cite{C}.
\end{remark}

\section{Indistinguishability principle}\label{sect:principles}

Now, let us combine the main theorem with the description of the universal index of tangles in a fixed surface obtained in~\cite{Nct}.

Let $D=D_1\cup\cdots\cup D_n$ be a tangle diagram in a surface $F$. Let $v\in\V(D)$ be a crossing of the diagram $D$.

The \emph{component type} $\tau(v)$ of the crossing $v$ is the pair $(i,j)$ where $D_i$ is the overcrossing component at $v$ and $D_j$ is the undercrossing component at $v$ (Fig.~\ref{pic:crossing_type}).

\begin{figure}[h]
\centering\includegraphics[width=0.1\textwidth]{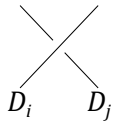}
\caption{A crossing of type $(i,j)$}\label{pic:crossing_type}
\end{figure}

Let $v$ be a self-crossing of a long component $D_i$ and $s_i$ be the starting point of $D_i$. Define the \emph{order type} $o$ of the crossing $v$ to be $o(v)=-1$ if $v$ is an \emph{early undercrossing}, i.e. one encounters first the undercrossing of $v$ while moving along $D_i$ from the point $s_i$, and $o(v)=+1$ if $v$ is an \emph{early overcrossing}, i.e. one encounters the overcrossing of $v$ first.

The following statements were proved in~\cite{Nct}.
\begin{theorem}\label{thm:universal_classical_index}
1. There is no nontrivial indices in the theory of classical knots $\mathfrak{CK}$.

2. The order type is the universal index in the theory of long classical knots  $\mathfrak{LCK}$.

3. The component type $\tau$ is the universal index in the theory of classical links $\mathfrak{CL}_n$ with $n$ (enumerated) components.
\end{theorem}

Then by Proposition~\ref{prop:basic_universal_index} the sign function is the universal signed index on $\mathfrak{CK}$; the pair $(o, sgn)$ is the universal signed index (with coefficients in $\Z_2^2$) on $\mathfrak{LCK}$, and $(\tau,sgn)$ is the universal sign index (with coefficients in $\{1,\dots,n\}^2\times\{-1,+1\}$) on $\mathfrak{CL}_n$. By Theorem~\ref{thm:main_theorem} the above mentioned maps are the universal traits in the corresponding diagram categories. We can reformulate this fact in the form of indistinguishability principle.

\begin{theorem}[Principle of indistinguishability]\label{thm:indistinguishability_principle}
\begin{enumerate}
  \item (for classical knots) No crossings of the same sign in a classical knot diagram can be distinguished by a trait.
  \item (for long classical knots) No crossings of the same sign and order type in a classical long knot diagram can be distinguished by a trait.
  \item (for classical links) No crossings of the same sign and component type in a classical link diagram can be distinguished by a trait.
\end{enumerate}
\end{theorem}

Also, we can reformulate the previous theorem as a \emph{substitution principle}. A manifestation of this principle is shown in Fig.~\ref{pic:crossing_substitute}.
\begin{theorem}[Substitution principle for classical knots]\label{thm:substitution_principle}
Let $v$ and $w$ be crossing of the same sign in a classical knot diagram $D$. Then there exist a sequence of Reidemeister moves which transforms $D$ to itself and moves the crossing $v$ to the crossing $w$.
\end{theorem}
Analogous statements can be formulated for long classical knots and classical links.

\begin{remark}
Analogues of the principles of indistinguishability and substitution can be formulated for tangles in a fixed surface: if two crossings in a tangle diagram have the same component, order and homotopy types (see~\cite{Nct}) then one of them can substitute for the other, and the crossings can not be distinguished.
\end{remark}

\begin{remark}
There is an analogue of substitution principle for \emph{arcs} of classical knot diagrams. By an arc we mean a part of a knot diagram which ends at crossings and contains no crossings inside.
\end{remark}

\begin{theorem}[Substitution principle for arcs of classical knot diagram]\label{thm:subsitution_classical_arc}
Let $D$ be a classical knot diagram in the sphere $S^2$ and $a,a'$ be two arcs of $D$. Then there is a sequence of isotopies and Reidemeiester moves which transforms $D$ to itself and transforms the arc $a$ to the arc $a'$.
\end{theorem}
\begin{proof}
  Choose arbitrary points $A\in a$ and $A'\in a'$. Rotate the diagram $D$ on the sphere so that $A$ moves to $A'$ and $a$ moves to an arc tangent to $a'$. Denote the diagram after the rotation by $D'$. Choose the point $A'$ for the infinity point in $S^2=\mathbb C\cup\{\infty\}$. The long knot diagrams $D\setminus\{A'\}$ and $D'\setminus\{A'\}$ define one long knot. Thus, there is a sequence of isotopies and Reidemeister moves in $\mathbb C$ which transforms $D'$ to $D$. Then the composition of the initial rotation and this sequence is the required chain of moves.
\end{proof}

\begin{remark}
The analogous statement for \emph{plane} classical diagrams (diagrams in $\R^2$) is wrong. One can construct a winding number-like invariant to distinguish the arcs of plane diagrams.
\end{remark}


So we see that the hope to use some inherent information on the crossings for construction of new invariants of classical knot is false. There are several ways to deal with this:
\begin{itemize}
\item Consider \emph{statistical} invariants which treat all crossings equally, like Jones polynomial, Khovanov homology and all the other combinatorial knot invariants.
\item Use not invariant but \emph{equivariant} labels for crossings. We can require that the labels of the crossings remain constant up to some isomorphisms. An example of invariants of this type are parity functors~\cite{Npf} and (bi)quandles~\cite{EN}. The substitution principle for arcs does not prevent the existence of coloring invariants.
\item Use nonreidemeister descriptions of knots. There are different decompositions of knots (into prime knots or JSJ decomposition). May be some of them can be catched with combinatorial elements of knots other that crossings (for example, quadrisecants, petal diagrams etc.)
\item There still remain virtual knots where indices and parities proved their worth.
\end{itemize}

\bibliographystyle{amsplain}

\appendix

\section{Examples of indices}

\subsection{Basic indices}

\begin{examplap}[The sign]\label{ex:sign_index}
The sign function is a signed index with coefficients in $\Z_2=\{-1,+1\}$ for the diagram categories of tangles in a surface and virtual tangles and all their diagram subcategories.
\end{examplap}

\begin{remarkap}
Often an index can be defined first as a signed index $\sigma$ on flat tangles and then for virtual tangles, it defines an index $\hat\sigma$ using the sign function (see Proposition~\ref{prop:sign_and_index}). This remark can be applied to the following two basic indices.
\end{remarkap}

\begin{examplap}[Component index]\label{ex:component_index}
Let $D=D_1\cup\cdots\cup D_n$ be a diagram of an oriented flat tangle (or a flat tangle in a fixed surface), and $v$ be a crossing of $D$. Then $v$ is an intersection point of components $D_i$ and $D_j$. We order the component according to the orientation as shown in Fig.~\ref{pic:component_type} left. The \emph{flat component (signed) index} of the crossing $v$ is the ordered pair $\tau^f(v)=(i,j)$. The flat component index $\tau^f$ is a signed index with coefficients in $\{1,\dots,n\}^2$.

\begin{figure}[h]
\centering\includegraphics[width=0.4\textwidth]{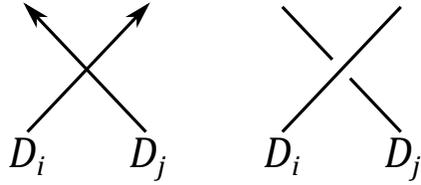}
\caption{The flat component index (left) and the component index (right)}\label{pic:component_type}
\end{figure}

For a diagram $D=D_1\cup\cdots\cup D_n$ of an oriented tangle, the \emph{component index} $\tau(v)$ of a crossing $v\in\V(D)$ is equal to $(i,j)$ if the component $D_i$ overcrosses the component $D_j$ at $v$ (Fig.~\ref{pic:component_type} right). It is an index with coefficients in $\{1,\dots,n\}^2$ on virtual tangles or tangles in a fixed surface. Note that $\tau=\hat\tau^f$.
\end{examplap}

\begin{remarkap}
On diagram of free tangles one can define a component index with coefficients in the set of unordered pairs $\{i,j\}$, $i,j\in\{1,\dots,n\}$.
\end{remarkap}

\begin{examplap}[Order index]\label{ex:order_index}
Let $D$ be a diagram of a flat long knot and $v$ be a crossing of $D$. The oriented smoothing at $v$ splits the diagram into halves $D^l_v$ and $D^r_v$ (Fig~\ref{pic:knot_halves}). Since $D$ is a long knot diagram, one of the halves is closed. Denote in by $D^c_v$. We define the \emph{flat order (signed) index} $o^f(v)$ of the crossing $v$ to be equal to $+1$ if $D^c_v=D^l_v$ and to be equal to $-1$ if $D^c_v=D^r_v$ (Fig.~\ref{pic:flat_order_index}). It is a signed index with coefficients in $\Z_2=\{-1,+1\}$ on flat knot diagrams.

\begin{figure}[h]
\centering
  \includegraphics[width=0.4\textwidth]{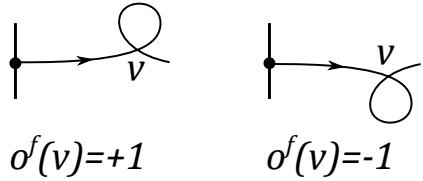}
  \caption{Flat order index}\label{pic:flat_order_index}
\end{figure}

On diagrams of long virtual knot, one can consider the \emph{order index} $o=-\hat o^f$. It is an index with coefficients in $\Z_2$. The order index discern the early undercrossings ($o=-1$) and the early overcrossings ($o=+1$).
\end{examplap}

\begin{remarkap}
Flat component and order signed indices can play the role of the sign function for flat knots and tangles in some situations.
\end{remarkap}

\subsection{Gluing index}

The gluing index appeared implicitly in the paper of A. Henrich~\cite{Henrich}, and a more general signed gluing index was used by P. Cahn in~\cite{C}.

\begin{definitionap}
Let $\mathfrak K$ be a diagram category of flat tangles (or flat tangles in a fixed surface). We define the \emph{signed singular diagram category} $\mathfrak K_{ss}$ as follows. The objects of $\mathfrak K_{ss}$ are triples $(D,v,\epsilon)$ where $D\in\mathfrak K$, $v\in\V(D)$ and $\epsilon\in\{-1,+1\}$. In other words, the diagrams of $\mathfrak K_{ss}$ are \emph{signed singular diagrams}, i.e. flat tangle diagrams with one singular crossing (``glued'' from a flat crossing $v$) marked with a sign $\epsilon=\pm1$.

The moves on signed flat diagrams are the moves which don't involve the singular crossing, and the second and third Reidemeister moves shown in Fig.~\ref{pic:signed_singular_reidemeister}.
\begin{figure}[h]
\centering\includegraphics[width=0.6\textwidth]{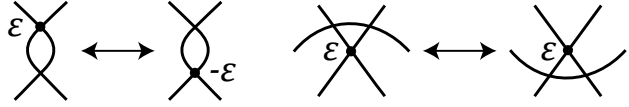}
\caption{Reidemeister moves with the signed singular crossing}\label{pic:signed_singular_reidemeister}
\end{figure}

Equivalence classes of signed singular diagrams modulo moves on them are called \emph{signed singular (flat) tangles}.
\end{definitionap}

\begin{examplap}[Gluing index]\label{ex:gluing_index}
Let $\mathcal{S}$ be the set of signed singular flat tangles. There is an involution on $\mathcal S$ which changes the sign of the singular crossing to the opposite. For a flat tangle diagram $D$ and a crossing $v\in\V(D)$ define its \emph{gluing signed index} $\sigma_{gl}(v)\in\mathcal S$ to be the signed singular flat tangle that corresponds to the diagram $(D,v,+1)$ with the singular crossing $v$ and the sign of the crossing $+1$. Then $\sigma_{gl}$ is a signed index on flat tangle diagrams with coefficients in $\mathcal S$. In fact, $\sigma_{gl}$ is the universal signed index on the flat tangle diagrams.

The linking invariant $lk_{\sigma_{gl}}$ of the signed index  is the operator $\mu$ in~\cite{C}.

The correspondent index $\bar\sigma_{gl}$ takes values in the tangles with one singular crossing (there is no any signs on the crossing). This index is defined on diagrams of flat tangles, so on diagrams of virtual tangles. The linking invariant $lk_{\bar\sigma_{gl}}$ is essentially the gluing invariant $\mathbf G$ in~\cite{Henrich}.
\end{examplap}

The universality of the index $\sigma_{gl}$ means that other indices on flat tangles can be presented by compositions $\psi\circ\sigma_{gl}$ where $\psi$ are maps from $\mathcal{S}$ to some sets $I$, i.e. the maps $\psi$ are $I$-valued invariants of the signed singular flat tangles. Examples of such indices are given below.

\subsection{Smoothing index}

\begin{example}\label{ex:oriented_smoothing_index}
Let $\mathfrak K$ be the diagram category of oriented flat tangles (or oriented flat tangles in a fixed surface $F$) with numbered components. Let $\mathcal T$ be the set of oriented flat tangles (or oriented flat tangles in $F$), i.e. the set of equivalence classes of diagrams modulo morphisms.  Let $D\in\mathfrak K$ be a tangle diagram and $v$ be a crossing of $D$. Consider the \emph{oriented and unoriented smoothings} $D^{or}_v$ and $D^{un}_v$ of the diagram $D$ at the crossing $v$ (see Fig.~\ref{pic:smoothings}).
\begin{figure}[h]
\centering\includegraphics[width=0.6\textwidth]{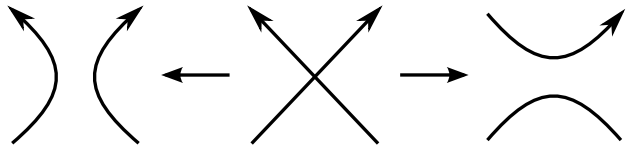}
\caption{Oriented (left) and unoriented (right) smoothings}\label{pic:smoothings}
\end{figure}
The smoothings $D^{or}_v$ and $D^{un}_v$ are diagrams of oriented tangles with numbered components. The oriented smoothing has natural orientation. The orientation of the unoriented smoothing is set forcedly by the orientation rule in Fig.~\ref{pic:smoothings} (the orientation of the lower arc is determined by the orientation of the upper one).

The \emph{oriented smoothing signed index} $\sigma_{sm}^{or}(v)$ is the tangle $[D^{or}_v]\in\mathcal T$ that corresponds to the oriented smoothing diagram. It is a signed index with coefficients in $\mathcal T$ where the involution on $\mathcal T$ swaps the numbers of the components containing the left and the right arcs on the smoothed diagram $D^{or}_v$ (involution is trivial if the arcs belong to one component).

The  \emph{unoriented smoothing signed index} $\sigma_{sm}^{un}(v)$ is the tangle $[D^{un}_v]\in\mathcal T$ that corresponds to the unoriented smoothing diagram. It is a signed index with coefficients in $\mathcal T$ where the involution on $\mathcal T$ changes the orientation of the smoothed component in $D^{un}_v$.
\end{example}

The index $[\tilde K^d_{smooth}]$ in~\cite{Henrich} and the index $\tilde L_c$ in~\cite{CGX} coincide with the index $\bar\sigma^{or}_{sm}$ for flat knots. The index $\tilde K_c$ in~\cite{CGX} coincides with the index $\bar\sigma^{un}_{sm}$ for flat knots.

Some indices can be deduced from the smoothed indices $\sigma_{sm}^{or}(v)$ and $\sigma_{sm}^{un}(v)$.

\begin{examplap}[Homological index]\label{ex:homological_index}
Let $\mathfrak K$ be the diagram category of a (flat) knot $K$ in a surface $F$. Let $D\in\mathfrak K$ and $v\in\V(D)$. The correspondence $v\mapsto [D^l_v]$ defines an oriented parity $hp$ with coefficients in $S=H_1(F,\Z)/[K]$ called the \emph{homological parity}~\cite{IMN,Npf}. Then $hp=\psi\circ\sigma_{sm}^{or}$ where the map $\psi$ from 2-component links to $S$ to is defined by the formula $\psi(D_1\cup D_2)=[D_1]$.

By definition the homological parity is a signed index, and if $K$ is a non flat knot we have the index $\widehat{hp}$ such that $\widehat{hp}(v)=[D^+_v]\in S$, $v\in\V(D)$. The linking invariant of $\widehat{hp}$ coincides with the small state sum $W_K$ in~\cite{F}.

The indices $f_\gamma$ in~\cite{CGX1} are obtained from the homology index using the map $\psi\colon S\to \Z$, $\alpha\mapsto \alpha\cdot\gamma$, where $\gamma\in H_1(F,\Z)$ and $D\cdot \gamma=0$.
\end{examplap}

The Gaussian index $n$ can be also deduced from the oriented smoothed index.

\begin{examplap}[Gaussian index]\label{ex:gaussian_index}
Let $\mathfrak K$ be the diagram category of flat knots. Let $D\in\mathfrak K$ and $v\in\V(D)$. The \emph{Gaussian index} $n(v)$ can be defined as $n(v)=W(\sigma^{or}_{sm}(v))$ where $W(\cdot)$ is the wriggle number~\cite{FK}. The wriggle number $W(D_1\cup D_2)\in\Z$ of a two-component link $D_1\cup D_2$ can be identified with the linking invariant $lk_{\tau}(D_1\cup D_2)$ of the component index on the link (the group $A_{\tau}$ is isomorphic to $\Z$ for two-component links).

The Gaussian index $n$  is a signed index with coefficients in $\Z$. For virtual knot diagrams, the corresponding index $\hat n=sgn\cdot n$ is the index $Ind$.
\end{examplap}

\begin{examplap}\label{ex:VKP_index}
The polynomials $F^n_D(t,\ell)$ in~\cite{KPV} rely on the combined index $\iota=(Ind,u_n\circ\sigma^{un}_{sm})$ where $u_n$ is the coefficient by the term $t^n$ of the (Gaussian) index polynomial for flat knots $u(t)\in\Z[t]$ (see~\cite{T}). The index $\iota$ takes values in the set $\Z\times\Z$.
\end{examplap}

\subsection{Based matrix index}

Based matrices were defined by V. Turaev in~\cite{T}. The based matrix of a knot diagram carries information about intersection numbers for halves of the diagram at the crossings. So based matrices can be considered as homological analogues of knot diagrams. Then an analogue of the gluing signed index $\sigma_{gl}$ are singed singular based matrices defined in~\cite{C}.

\begin{definitionap}\label{def:signed_singular_based matrix}
A \emph{signed singular based matrix} is a quadruple $T=(G,s,d,b,\epsilon)$ where $G$ is a finite set, $s,d\in G$, $s\ne d$,  $b\colon G\times G\to\Z$ is a skew-symmetric map and $\epsilon\in\{-1,+1\}$.

An element $g\in G\setminus\{s\}$ is \emph{annihilating} if $b(g,h)=0$ for any $h\in G$.

An element $g\in G\setminus\{s\}$ is \emph{core} if $b(g,h)=b(g,s)$ for any $h\in G$.

Two elements $g_1,g_2\in G\setminus\{s\}$ are \emph{complementary} if $b(g_1,h)+b(g_2,h)=b(g,s)$ for any $h\in G$.

We consider the following \emph{elementary extensions}:
\begin{itemize}
  \item[$M1$] Given a signed singular based matrix $T=(G,s,d,b,\epsilon)$, form the signed singular based matrix $\bar T=(\bar G,s,d,\bar b,\epsilon)$ such that $\bar G = G\sqcup\{g\}$, $\bar b|_{G\times G}=b$ and $\bar b(g,h)=0$ for all $h\in G$.
  \item[$M2$] Given $T=(G,s,d,b,\epsilon)$, form the signed singular based matrix $\bar T=(\bar G,s,d,\bar b,\epsilon)$ such that $\bar G = G\sqcup\{g\}$, $\bar b|_{G\times G}=b$ and $\bar b(g,h)=b(s,h)$ for all $h\in G$.
  \item[$M3$] Given $T=(G,s,d,b,\epsilon)$, form the signed singular based matrix $\bar T=(\bar G,s,d,\bar b,\epsilon)$ such that $\bar G = G\sqcup\{g_1,g_2\}$, $\bar b|_{G\times G}=b$, and $g_1$ and $g_2$ are complementary in $\bar T$, i.e.  $\bar b(g_1,h)+\bar b(g_2,h)=b(s,h)$ for all $h\in G$.
  \item[$N$] Given $T=(G,s,d,b,\epsilon)$ and an element $g\in G\setminus\{s\}$ such that $d$ and $g$ are complementary, form the signed singular based matrix $\bar T=(G,s,g,b,-\epsilon)$.
\end{itemize}

A signed singular based matrix $T$ is called \emph{primitive} if the inverse excisions $M1^{-1}, M2^{-1}, M3^{-1}$ can not be applied to $T$.

Two signed singular based matrix $T$ and $T'$ are \emph{homologous} if $T'$ can be obtained from $T$ by a sequence of elementary excisions $M1, M2, M3, N$ and the inverse transformations to them.
\end{definitionap}

The model example for a signed singular based matrix is the following. Let $(D,v,\epsilon)$ be a signed singular diagram of a flat knot. Then the quintuple $T(D,v,\epsilon)=(G,s,d,b,\epsilon)$ where $G=\{s\}\sqcup \V(D)$, $d=v$, $b(w,w')=[D^l_w]\cdot[D^l_{w'}]$, $w,w'\in\V(D)$, is the intersection number of the knot halves, and $b(s,w)=[D]\cdot[D^l_w]$, $w\in\V(D)$, is a signed singular base matrix.

With the correspondence $(D,v,\epsilon)\mapsto T(D,v,\epsilon)$, Reidemeister moves on signed singular diagrams turn into moves $M1, M2, M3, N$.

\begin{propositionap}[\cite{C}]\label{prop:based_matrix}
1. If two signed singular diagrams $(D,v,\epsilon)$ and $(D',v',\epsilon')$ are equivalent then the signed singular based matrices $T(D,v,\epsilon)$ and $T(D',v',\epsilon')$ are homologous.

2. For any signed singular based matrix $T$ there exist a primitive signed singular based matrix $T_\bullet=(G_\bullet,s,d,b_\bullet,\epsilon)$ which is homologous to $T$. The based matrix $T_\bullet$ is unique up to isomorphisms, the move $N$, and the replacement of the core element $d$ by an annihilating element together with simultaneous change of $\epsilon$ by $-\epsilon$ (in the case when $d$ is a core element).
\end{propositionap}

Thus, $T_\bullet(D,v,\epsilon)$ is an invariant of signed singular based matrices (up to isomorpisms and the moves described in Proposition~\ref{prop:based_matrix}).

\begin{examplap}[Based matrix index]\label{ex:based_matrix_index}
Let $\mathfrak {FK}$ be the diagram category of flat knots. Let $D\in\mathfrak {FK}$ and $v\in\V(D)$. The \emph{based matrix signed index} $\sigma_{bm}(v)$ of the crossing $v$ is defined by the formula $\sigma_{bm}(v)=T(\sigma_{gl}(v))=T_\bullet(\sigma_{gl}(v))\in\mathcal M$ where $\mathcal M$ is the set of homology classes of signed singular based matrices ($\mathcal M$ is also equal to the set of primitive signed singular based matrices up to moves described in Proposition~\ref{prop:based_matrix}). The involution on $\mathcal M$ changes the sign $\epsilon$ of a signed singular based matrix to the opposite.

The correspondent index $\bar\sigma_{bm}(v)$ appears in~\cite{Henrich}.
\end{examplap}

\begin{remarkap}
The reduced stable parity on based matrices in~\cite{Nbm} is obtained by applying the intersection formula~\cite{Nif} to an index which can be deduced from the index $\bar\sigma_{bm}(v)$.
\end{remarkap}


For virtual knots, the definition of based matrices can be refined by involving the sign function~\cite{T1}. A \emph{graded signed singular based matrix} is a quadruple $(G,s,d,b,sgn)$ where $G$ is a set, $s,d\in G$, $b\colon G\times G\to \Z$ is a skew-symmetric form and $sgn$ is a function from $G\setminus\{s\}$ to $\{-1,+1\}$. The sign of the singular based matrix is $\epsilon=sgn(d)$.

The refined definition of complementary elements are defined as follows: $g_1,g_2\in G\setminus\{s\}$ are \emph{complementary} if $sgn(g_1)=-sgn(g_2)$ and $b(g_1,h)+b(g_2,h)=b(s,h)$ for all $h\in G$. The refined elementary excisions $M3$ and $N$ use the new definition of complementarity.

Given a diagram $D$ of a virtual knot and a crossing $v\in\V(D)$, the definition of the signed singular based matrix $T(D,v,sgn(v))=(G,s,d,b,sgn(v))$ extends to the graded signed singular based matrix $T_{gr}(D,v)=(G,s,d,b,sgn)$.

With the sign function in a graded signed singular based matrix $T=(G,s,d,b,sgn)$, one can define the forms $b^{\alpha\beta}\colon G\times G\to \Z$, $\alpha,\beta\in\{-1,+1\}$, by the formulas
\begin{gather*}
b^{\alpha\beta}(g,h)=\alpha\beta sgn(g)sgn(h) b(g,h)-\frac{1-\alpha sgn(g)}2 b(g,s)-\frac{1-\beta sgn(h)}2 b(s,h),\\
b^{\alpha\beta}(g,s)=\alpha sgn(g) b(g,s),\quad b^{\alpha\beta}(s,h)=\beta sgn(h) b(s,h),\quad b^{\alpha\beta}(s,s)=0,
\end{gather*}
for $g,h\in G\setminus\{s\}$. For complementary elements $g_1,g_2\in G$ of $T$ we have the relation $b^{\alpha\beta}(g_1,h)=b^{\alpha\beta}(g_2,h)$ for any $h\in G$ and any $\alpha,\beta\in\{-1,+1\}$. Hence, the $\iota^{\alpha\beta}(g)=b^{\alpha\beta}(d,g)$ is an index with coefficients in $\Z$ on the category of graded signed singular based matrices with excision moves. Let $p^{\alpha\beta}(T)=lk_{\iota^{\alpha\beta}}(T)\in\Z[\Z]$ be the corresponding linking invariant.

\begin{examplap}[Intersection index]\label{ex:intersection_index}
Let $\mathfrak {VK}$ be the virtual knot theory. Let $D\in\mathfrak {VK}$ be a diagram and $v$ be a crossing in $D$. The \emph{intersection index} $i^{\alpha\beta}(v)$ of the crossing $v$ is defined as $i^{\alpha\beta}(v)=p^{\alpha\beta}\circ T_{gr}(D,v)\in\Z[\Z]$. The the linking invariants of the indices $i^{\alpha\beta}$, $\alpha,\beta\in\{-1,+1\}$ can be identified (up to terms with the loop index values) with the polynomials $f_{ij}$ in~\cite{HNNS}.
\end{examplap}

\subsection{Secondary index}

Let us recall the definition of the oriented parity~\cite{Nwp,Npf}.
\begin{definitionap}\label{def:oriented_parity}
Let $\mathfrak K$ be a diagram category. An \emph{oriented parity} $p$ on $\mathfrak K$ is a signed index with coefficients in an abelian group $A$ (with the involution $x^*=-x$, $x\in A$) which obeys the following conditions:
\begin{itemize}
  \item[(P0)] if a first Reidemeister move can be applied to a crossing $v$ in a diagram $D\in\mathfrak K$ then $p(v)=0$.
  \item[(P3+)] if $f\colon D\to D'$ is a third Reidemeister move then
\[
\epsilon_\Delta(v_1)\cdot p_D(v_1)+\epsilon_\Delta(v_2)\cdot p_D(v_2)+\epsilon_\Delta(v_3)\cdot p_D(v_3)=0
\]
where $v_1,v_2,v_3$  are the crossings involved in the move and $\epsilon_\Delta(v_i)$ is the incidence index of the crossing $v_i$ to the disappearing triangle $\Delta$, see Fig.~\ref{pic:incidence_index}.
\begin {figure}[h]
\centering
\includegraphics[width=0.1\textwidth]{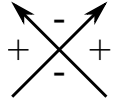}
\caption{Incidence indices}\label{pic:incidence_index}
\end {figure}
\end{itemize}
\end{definitionap}

An examples of oriented parities are the homological parity (Example~\ref{ex:homological_index}) and the Gaussian index $n$ (Example~\ref{ex:gaussian_index}).

\begin{examplap}[Secondary index]
Let $p$ be an oriented parity with coefficients in an abelian group $A$ on the diagram category $\mathfrak {FK}$ of flat knots. Let $\tilde A = \bigoplus_{a\in A}\Z[A/\langle a\rangle]$, and $\ast$ be the involution on $\tilde A$ such that $(A_a)^\ast=A_{-a}$, $a\in A$, and
\[
\left(\sum_{x\in A/\langle a\rangle}\lambda_x\cdot [x]\right)^\ast=-\sum_{x\in A/\langle a\rangle}\lambda_x\cdot [-x],\quad a\in A.
\]
For a diagram $D\in\mathfrak {VK}$ and a crossing $v\in\V(D)$ define the \emph{secondary index} of the crossing $v$ by the formula
\[
\sigma_p(v)=\sum_{v'\in\V(D)}lk(v,v')\cdot [lk(v,v')\cdot p(v')]\in \Z[A/\langle p(v)\rangle]\subset \tilde A
\]
where the number $lk(v,v')$ is defined a shown in Fig.~\ref{pic:linking_number} (the diagram $D$ is considered as a Gauss diagram; the chords in flat knot diagrams are oriented as if they were positive crossings). We set $lk(v,v)=0$ for any $v\in\V(D)$.
\begin {figure}[h]
\centering
\includegraphics[width=0.45\textwidth]{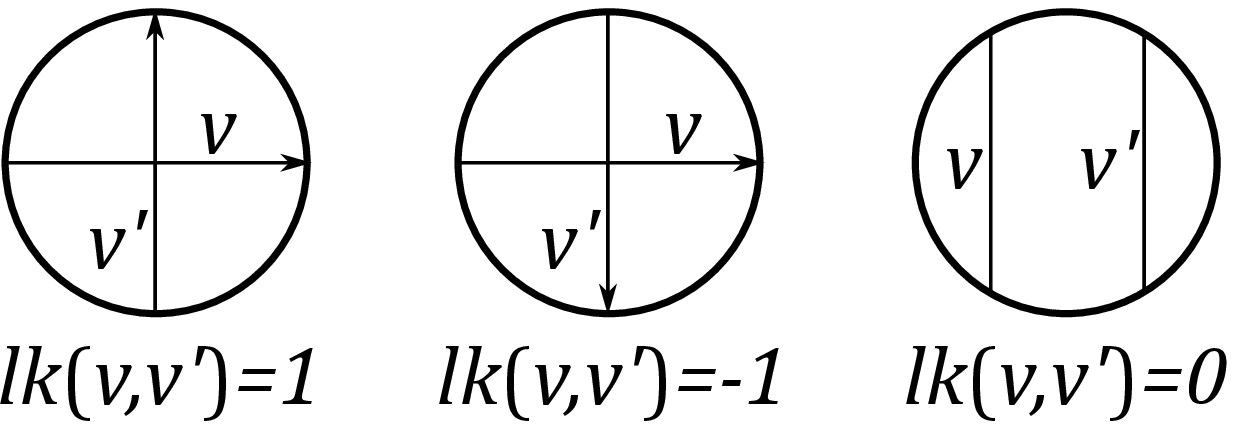}
\caption{The number $lk(v,v')$}\label{pic:linking_number}
\end {figure}
\end{examplap}

\begin{propositionap}\label{porp:secondary_index}
The map $\sigma_p$ is a signed index on $\mathfrak {FK}$ with coefficients in $\tilde A$.
\end{propositionap}
\begin{proof}
  Let us check the property (I0). Let $D\in\mathfrak {FK}$, $v\in\V(D)$ and $f\colon D\to D$ be an increasing second Reidemeister move with new crossings $w_1,w_2\in\V(D')$. Then $p(w_1)=-p(w_2)$ and $lk(f_*(v),w_1)=-lk(f_*(v),w_2)$. Then
\begin{multline*}
\sigma_p(f_*(v))=\sum_{w'\in\V(D')}lk(f_*(v),w')\cdot [lk(f_*(v),w')\cdot p(w')]=\\
\sum_{w\in\V(D)}lk(f_*(v),f_*(w))\cdot [lk(f_*(v),f_*(w))\cdot p(f_*(w))]+\\
lk(f_*(v),w_1)\cdot [lk(f_*(v),w_1)\cdot p(w_1)]+lk(f_*(v),w_2)\cdot [lk(f_*(v),w_2)\cdot p(w_2)]=\\
\sum_{w\in\V(D)}lk(v,w)\cdot [lk(v,w)\cdot p(w)]=\sigma_p(v).
\end{multline*}

Let $f\colon D\to D'$ be a third Reidemeister move shown in Fig.~\ref{pic:secondary_index_r3}. Then $lk(v,w)=1$, $lk(u',v')=-1$ and $lk(u',w')=1$. By the property (P3+) $p(u)-p(v)-p(w)=0$.

\begin {figure}[h]
\centering
\includegraphics[width=0.3\textwidth]{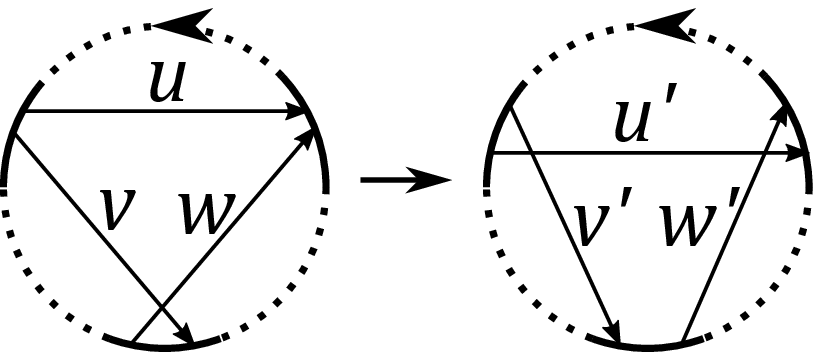}
\caption{The number $lk(v,v')$}\label{pic:secondary_index_r3}
\end {figure}
Then
\begin{gather*}
\sigma_p(f_*(u))-\sigma_p(u)=-[-p(v')]+[p(w')]=0,\\
\sigma_p(f_*(v))-\sigma_p(v)=[p(u')]-[p(w)]=0,\\
\sigma_p(f_*(u))-\sigma_p(u)=-[-p(u')]+[-p(v)]=0,
\end{gather*}
because $-p(u')=-p(u)\equiv p(v) \mod p(w)$ etc. The other cases of third Reidemeister moves are proved analogously.

Let us check the property (I2+). Let $v_1,v_2\in\V(D)$, $D\in\mathfrak {FK}$, be crossings to which a second Reidemeister move can be applied. Then $p(v_1)= -p(v_2)\equiv 0\mod p(v_i)$, $i=1,2$; $lk(v_1,v_2)=-lk(v_2,v_1)$, and  $lk(v_1,v)=-lk(v_2,v)$ for any $v\in\V(D)\setminus\{v_1,v_2\}$. Thus, $\sigma_p(v_2)=\sigma_p(v_1)^\ast$.
\end{proof}

The transcendental index polynomials in~\cite{Chengtr, CGXtr,CGX1} rely on the secondary index $\iota_p=\hat\sigma_p$ for some oriented parities $p$.

\subsection{Derived parities}

A series of signed indices (in fact, oriented parities) on virtual knot diagrams was constructed in~\cite{Nif}.

\begin{examplap}[Derived parities]\label{ex:derived_parities}
Let $\mathfrak K=\mathfrak K(K)$ be the diagram category of a virtual knot $K$, and $p$ be an oriented parity on $\mathfrak K$ with coefficients in an abelian group $A$. Then the \emph{derived parity} $p'$ on the diagram category $\mathfrak K$ is given by the formula
\[
p'(v)=\sum_{v'\in\V(D)}p(v')\cdot(D^l_v\cdot D^-_{v'})\in A', \quad v\in\V(D), D\in\mathfrak K,
\]
where $A'=A/\langle\sum_{v\in\V(D)}p(v)\cdot Ind(v)\rangle$. By definition $p'$ is a signed index with coefficients in $A'$.

One can define also the higher derivatives $p''$, $p'''$ and so on.
\end{examplap}

\subsection{Indices via parity projection}

Recall the definition of a weak parity~\cite{M3}.

\begin{definitionap}\label{def:weak_parity}
Let $\mathfrak K$ be a diagram category. A \emph{weak parity} $\psi$ on the diagram category $\mathfrak K$ is an index with coefficients in $\Z_2$ which satisfies the condition
\begin{itemize}
  \item[$(\Psi3)$]  if $f\colon D\to D'$, $D,D'\in\mathfrak K$, is a third Reidemeister move on crossings $u,v,w\in\V(D)$ and $\psi(u)=\psi(v)=0$ then $\psi(w)=0$.
\end{itemize}

Any weak parity $\psi$ induces a natural transformation $\Psi$ (called the \emph{parity projection} induced by $\psi$) from the diagram category $\mathfrak K$ to the diagram category of virtual tangles (category of flat/free tangles if $\mathfrak K$ consists of diagrams of flat/free tangles) as follows: given a diagram $D\in\mathfrak K$, the diagram $\Psi(D)$ is obtained by replacing all classical crossings $v\in\V(D)$ such that $\psi(v)=1$ by virtual crossing. The parity projection $\Psi$ defines a correct map from the knots, links or tangles of $\mathfrak K$ (i.e. the equivalence classes of diagrams to $\mathfrak K$) to virtual (or flat or free) tangles.
\end{definitionap}

\begin{propositionap}[\cite{Nwp}]
Let $p$ be an ordered parity on $\mathfrak K$ with coefficients in an abelian group $A$. Then the map $\psi_p$ given by the formula
\[
\psi_p(v)=\left\{\begin{array}{cc}
                   1, & p(v)\ne 0, \\
                   0, & p(v)=0,
                 \end{array}\right.
\]
$v\in\V(D)$, $D\in\mathfrak K$, is weak parity on $\mathfrak K$.
\end{propositionap}

\begin{examplap}[Index induced by a parity projection]
Let $\mathfrak D$ be a diagram category, $\psi$ be a weak parity on $\mathfrak K$, $\Psi$ be the correspondent parity projection to the diagram category $\mathfrak {VT}$ of virtual tangles, and $\iota$ be an index on $\mathfrak {VT}$ with coefficients in a set $I$. The the \emph{index induced from $\iota$ by the parity projection $\Psi$} is the index $\psi(\iota)$ on $\mathfrak K$ with coefficients in $I\sqcup\{\bullet\}$ defined by the formula
\[
\psi(\iota)(v)=\left\{\begin{array}{cc}
                   \bullet, & \psi(v)\ne 0, \\
                   \iota(\Psi_*(v)), & p(v)=0,
                 \end{array}\right.
\]
$v\in\V(D)$, $D\in\mathfrak K$.
\end{examplap}

The index polynomials in~\cite{IK17,IKL18,IKL19,IKL20,J} use the indices induced from the Gaussian index $Ind$ by the weak parity $\psi_{Ind_n}$ where $Ind_n=Ind \mod n$, $n\in\N\cup\{0\}$ is the Gaussian index considered modulo $n$.

\subsection{Biquandle index}

Recall that a \emph{biquandle}~\cite{EN} is a set $B$ with two operations $\ast,\circ\colon B\times B\to B$ which satisfy the following conditions:
\begin{enumerate}
\item $x\circ x=x\ast x$ for all $x\in B$;
\item the maps $(x,y)\mapsto (y,x\circ y)$, $(x,y)\mapsto (x,y\ast x)$ and $(x,y)\mapsto (x\circ y, y\ast x)$ are bijections of $B\times B$
\item for any $x,y,z\in B$;
\begin{gather*}
(x\circ y)\circ (z\circ y)=(x\circ z)\circ (y\ast z),\\
(x\circ y)\ast (z\circ y)=(x\ast z)\circ (y\ast z),\\
(x\ast y)\ast (z\ast y)=(x\ast z)\ast (y\circ z).
\end{gather*}
\end{enumerate}

For a diagram $D$ of a virtual tangle, a \emph{colouring} of $D$ with a biquandle $B$ is a map $c$ from the set of the arcs in $D$ to $B$ which obeys the colouring rule in Fig.~\ref{pic:biquandle_operations} for any crossing of $D$.

\begin{figure}[h]
\centering\includegraphics[width=0.4\textwidth]{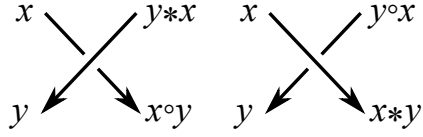}
\caption{The colouring rule}\label{pic:biquandle_operations}
\end{figure}

Let $Col_B(D)$ denote the set of colourings of the diagram $D$. For any diagram $D'$ of the same tangle, there is a bijection between the colouring sets $Col_B(D)$ and $Col_B(D')$.

Let us describe the (signed) index introduced in~\cite{Cheng2}.

\begin{examplap}[Biquandle index]
Let $B$ be a biquandle. Consider the set $\tilde B=\tilde B^-\sqcup\tilde B^+$ where $B^\pm$ are the quotient sets
\begin{gather*}
  \tilde B^+=B/\langle (x,y)=(x\circ z,y\circ z)=(x\ast z,y\ast z)=(x\circ z,y\ast z),\ x,y,z\in B\rangle, \\
  \tilde B^-=B/\langle (x,y)=(x\circ z,y\circ z)=(x\ast z,y\ast z)=(x\ast z,y\circ z),\ x,y,z\in B\rangle,
\end{gather*}
with the involution $\ast\colon \tilde B^\pm\to\tilde B^\mp$, $(x,y)^\ast=(y,x)$, $x,y\in B$.

Let $\mathfrak K$ be a diagram category of virtual tangles. Let $D\in\mathfrak K$ and $Col_B(D)$ be the set of colourings of $D$. For any colouring $c\in Col_B(D)$ and any crossing $v\in\V(D)$ consider the element $\sigma_{B,c}(v)=(x,y)\in\tilde B^{sgn(v)}$ where $x,y$ are the colours of arcs incident to $v$ (see Fig.~\ref{pic:biquandle_operations}). The \emph{biquandle signed index} $\sigma_B(v)$ of the crossing $v$ is the multiset $\sigma_B(v)=\{\sigma_{D,c}(v)\}_{c\in Col_B(D)}$.
\end{examplap}

From the biquandle properties, the following statement holds.

\begin{propositionap}
The map $\sigma_B$ is a signed index on $\mathfrak K$ with coefficients in the multi-subsets of $\tilde B$.
\end{propositionap}

\begin{remarkap}\label{rem:biquandle_Gaussian_index}
The map $\sigma_{B,c}$ can be considered as a \emph{local signed index}. It satisfies the property (I2+) and the following property: for any morphism $f\colon D\to D'$ and any crossing $v\in dom(f_*)$ we have $\sigma_{B,f_*(c)}(f_*(v))=\sigma_{B,c}(v)$ where $f_*(c)\in Col_B(D')$ is the colouring which corresponds to the colouring $c$. Note that $\sigma_{B,c}$ may not define a (global) signed index due to the colouring monodromy.
\end{remarkap}

\begin{examplap}
Let $B=\Z$ be the biquandle with operations $x\circ y=x\ast y=x+1$. Then $\tilde B^\pm\simeq \Z$ by the isomorphism $(x,y)\mapsto x-y$. Then the index $\sigma_{B,c}$ coincides with the Gaussian signed index $n$ on the virtual knots.
\end{examplap}

\subsection{An example}

Consider the knot 3.1 from Green's table~\cite{Green}. Let us calculate some of the indices described above.

\begin{figure}[h]
\centering\includegraphics[width=0.11\textwidth]{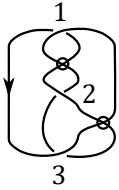}
\caption{The knot 3.1}\label{pic:knot31}
\end{figure}

\begin{enumerate}
  \item The gluing signed index $\sigma_{gl}$ values of the crossings are shown in Fig.~\ref{pic:31gluing_index}
\begin{figure}[h]
\centering\includegraphics[width=0.4\textwidth]{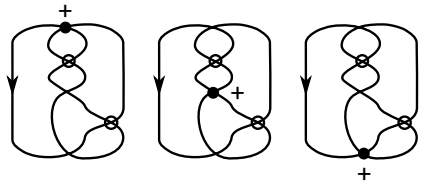}
\caption{The gluing index $\sigma_{gl}$ on the knot 3.1}\label{pic:31gluing_index}
\end{figure}
  \item The values of the smoothing signed indices $\sigma_{sm}^{or}$ and $\sigma_{sm}^{un}$ on the crossings are shown in Fig.~\ref{pic:31smoothing_index}
\begin{figure}[h]
\centering\includegraphics[width=0.4\textwidth]{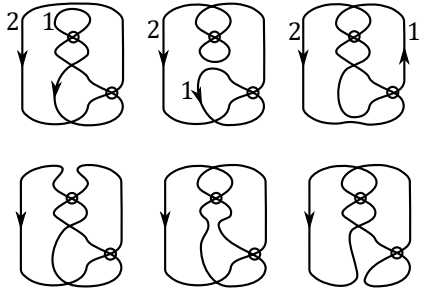}
\caption{The smoothing signed indices $\sigma_{sm}^{or}$ (top) and $\sigma_{sm}^{un}$ (bottom) on the knot 3.1}\label{pic:31smoothing_index}
\end{figure}
  \item The values of the based matrix sign index $\sigma_{bm}$ on the crossings are

\[
\left(\begin{array}{r|rrr}
        0 & \mathbf{-1} & -1 & 2 \\
        \hline
        \mathbf 1 & \mathbf 0 & \mathbf 0 & \mathbf 2 \\
        1 & \mathbf 0 & 0 & 1 \\
        -2 & \mathbf{-2} & -2 & 0
      \end{array}\right)^-
      \
      \left(\begin{array}{r|rrr}
        0 & -1 & \mathbf{-1} & 2 \\
        \hline
        1 & 0 & \mathbf 0 & 2 \\
        \mathbf 1 & \mathbf 0 & \mathbf 0 & \mathbf 1 \\
        -2 & -2 & \mathbf{-2} & 0
      \end{array}\right)^+
      \
      \left(\begin{array}{r|rrr}
        0 & -1 & -1 & \mathbf 2 \\
        \hline
        1 & 0 & 0 & \mathbf 2 \\
        1 & 0 & 0 & \mathbf 1 \\
        \mathbf{-2} & \mathbf{-2} & \mathbf{-2} & \mathbf 0
      \end{array}\right)^-
\]
The matrix is the matrix of the form $b$; the first column and row correspond to the element $s$; the column and the row which correspond to $d$ are highlighted with the bold font; the sign $\epsilon$ is the matrix superscript.

\item The Gaussian index is $n(1)=-1$, $n(2)=-1$, $n(3)=2$.
\item The secondary index of the Gaussian index is trivial for the first and the second crossings, and takes the value $\sigma_n(3)=2\cdot[1]\in\Z[\Z_2]$ for the third crossing.
\item The derived parity $n'$ of the Gaussian index has coefficients in $\Z_4$ and is equal to $n'(1)=-1$, $n'(2)=1$, $n'(3)=-1$. The second derivative $n''$ has coefficients in $\Z_4$ and is equal to  $n''(1)=-1$, $n''(2)=0$, $n''(3)=1$.
\item Consider the weak parity $\psi=\psi_{Ind_2}$. Then the induced index $\iota=\psi(Ind)$ is equal $\iota(1)=\iota(2)=\bullet$, $\iota(3)=0$.
\item Let $B=\Z$ be the biquandle with operations $x\circ y=x\ast y=z+1$. Then by Remark~\ref{rem:biquandle_Gaussian_index} the biquandle signed index $\sigma_B$ is isomorphic to the Gaussian index $n$ and takes the values $-1,-1,2$ on the crossings of the knot 3.1.
\end{enumerate}

\end{document}